\documentclass[10pt,reqno]{amsart}
\usepackage[toc,page]{appendix}
\usepackage{amsmath}
\usepackage{amsthm}
\usepackage{amsfonts}
\usepackage{amssymb}
\usepackage{graphicx}
\usepackage{bm}
\usepackage{flafter}
\usepackage{bm}
\usepackage[all]{xy}

\newtheorem{theorem}{Theorem}

\newtheorem{corollary}[theorem]{Corollary}

\newtheorem{definition}[theorem]{Definition}
\newtheorem{comments}[theorem]{Comments}

\newtheorem{examples}[theorem]{Examples}

\newtheorem{lemma}[theorem]{Lemma}
\newtheorem{notations}[theorem]{Notations}
\newtheorem{notation}[theorem]{Notation}
\newtheorem{proposition}[theorem]{Proposition}
\newtheorem{remark}[theorem]{Remark}

\renewenvironment{proof}{\textit{Proof}}{${}\hfill\square$}

\begin{document}

\title{Prolongations of convenient Lie algebroids}
\author{Patrick Cabau \& Fernand Pelletier}

\address{Unit\'e Mixte de Recherche 5127 CNRS, Universit\'e  de Savoie Mont Blanc, Laboratoire de Math\'ematiques (LAMA),Campus Scientifique,  73370 Le Bourget-du-Lac, France}
\email{patrickcabau@gmail.com, fernand.pelletier@univ-smb.fr}

\date{}
\maketitle

\begin{abstract} 
We first define the concept of Lie algebroid in the convenient setting. In reference to the finite dimensional context, we adapt the notion of prolongation of a Lie algebroid over a fibred manifold to a convenient Lie algebroid over a fibred manifold. Then we show that this construction is stable under projective and direct limits under adequate assumptions. 
\end{abstract}



\section{introduction}
In classical mechanics, the configuration space is a finite dimensional smooth manifold $M$ where the tangent bundle $p_M:TM \to M$ corresponds to the velocity space. This geometrical object plays a relevant role in the Lagrangian formalism between tangent and cotangent bundles (cf. \cite{Kle}). In \cite{Wei}, Weinstein develops a generalized theory of Lagrangian Mechanics on Lie algebroids. In \cite{Lib}, 
Libermann shows that such a formalism is not possible, in general, if we consider the tangent bundle of a Lie  algebroid. The notion of prolongation of a Lie algebroid  introduced by Higgins and Mackenzie \cite{HiMa} offers a nice context in which such a formalism was generalized by Mart\'inez (cf. \cite{Ma1} and \cite{Ma2}).\\

The notion of Lie algebroid in the Banach setting was simultaneously introduced in \cite{Ana} and \cite{Pe1}. Unfortunately, in this setting, there exist many problems to generalizing all canonical Lie structures on a finite dimensional Lie algebroid  and, at first, a nice  definition of a Lie bracket (cf. \cite{CaPe}). In this paper, we consider the more general convenient setting (cf. \cite{KrMi}) in which   we give a definition of a convenient  Lie algebroid (and more generally a partial convenient Lie algebroid) based on a precise notion of sheaf of Lie brackets
\footnote{cf. Remark \ref{R_LocalGlobal}} on a convenient anchored  bundle (cf. section \ref{___AlmostLieAlgebroidAndLieAlgebroid}).\\
As in finite dimension, given a convenient anchored bundle $ \left( \mathcal{A},\pi, M,\rho \right) $,
\footnote{The anchor $\rho$ is a vector bundle morphism $\rho: \mathcal{A} \to TM$.}
 the total space of the  prolongation $
\hat{\bf p}:\mathbf{T}\mathcal{M}\to \mathcal{M}$ of $ \left( \mathcal{A},\pi, M,\rho \right) $ over a fibred manifold $\mathbf{p}:\mathcal{M}\to M$, is the pullback over $\rho$ of the bundle 
$T\mathbf{p}: T\mathcal{M} \to TM$. Moreover, we have an anchor $\hat{\rho}: \mathbf{T}\mathcal{M} \to T\mathcal{M}$. \\
In finite dimension, the Lie bracket on $\mathcal{A}$ gives rise to a Lie bracket on  $\mathbf{T}\mathcal{M}$. Unfortunately, this is not any more true in infinite dimension. If $\tilde{\pi}:\widetilde{\mathcal{A}}\to \mathcal{M}$ is the pullback of $\pi:\mathcal{A}\to M$, over $\mathbf{p}$, then the module of local sections of $\widetilde{\mathcal{A}}$ is no longer finitely generated by local sections along $\mathbf{p}$. For this reason, the way to define the prolongation of the bracket does not work as in finite dimension. Thus the prolongation of such a Lie algebroid is not again a Lie algebroid but only a strong partial Lie algebroid (cf. $\S$ \ref{___StructureofPartialLieAlgebroid}). We also define (non linear)  connections on such a prolongation.  As for the tangent bundle of a Banach vector bundle (cf. \cite{AgSu}), we then show that the kernel bundle of  $\hat{\bf p}:\mathbf{T}\mathcal{M}\to \mathcal{M}$ is split if and only if there exists a linear connection on $\mathcal{M}$.\\
In the Banach setting, it is proved that if the kernel of the anchor $\rho$ of a Banach Lie algebroid $(\mathcal{A}, \pi, M,\rho, [.,.]_\mathcal{A})$ is split and its range is closed, then the associated distribution on $M$ defines a (singular) foliation (cf. \cite{Pe1}). Under these assumptions, we show that the prolongation $ \left( \mathbf{T}\mathcal{A}, \hat{\mathbf{p}}, \mathcal{A}, \hat{\rho} \right) $ has the same properties, and the foliation defined by $\hat{\rho}(\mathbf{T}\mathcal{A})$ on $\mathcal{A}$ is exactly  the set 
$\{\mathcal{A}_{| L}, \; L \textrm{ leaf of } \rho(\mathcal{A})\}$ (cf. Theorem \ref{T_tildefol}).\\
As an illustration of these notions,  in the convenient setting, and not only in a Banach one, we end this work by the prolongation of projective (resp. direct) limits of sequences of projective (resp. ascending) sequences of fibred Banach Lie 
algebroids with finite or infinite dimensional fibres. \\

\emph{This work can be understood as the basis for further studies on  how the  Lagrangian formalism  on finite dimensional Lie algebroid (cf. \cite{Ma2} for instance) can be generalized in this convenient framework.}\\

Section \ref{_ConvenientLieAlgebroid} is devoted to the presentation of the prerequisities needed in this paper about (partial)  convenient Lie algebroids. After some preliminaries on notations, we introduce the notion of convenient anchored bundle. Then we define a notion of  almost bracket on such a vector bundle. The definition of a  convenient Lie algebroid and some of its properties are given in subsection \ref{___AlmostLieAlgebroidAndLieAlgebroid}. The next subsection exposes the concept of partial convenient Lie algebroid. The following subsection is devoted to the definitions of some derivative operators (Lie derivative with respect to some sheaf of $k$-forms sections and exterior derivative), in particular of strong  partial Lie algebroids. The last subsection recall some results about integrability of Banach Lie algebroids when the Banach Lie algebroid is split and the range of its  anchor is closed (cf \cite{Pe2}).\\
The important part of this work is contained in section \ref{__ProlongationOfAConvenientLieAlgebroidAlongAFibration}. In the first subsection, we build the prolongation of a convenient  anchored bundle. Then, in subsection \ref{___ProlongationOfTheLieBracket}, we can define a Lie bracket on local projectable sections of the total space of the prolongation. We then explain why this bracket cannot be extended to the whole set of local sections if the typical fibre of the anchored bundle is not finite dimensional, which is the essential difference with the finite dimensional setting. We end this section by showing that if a Banach Lie algebroid is split and the range of its anchor is closed, the same is true for its prolongation and the range of the anchor of the prolongation defines also a foliation even if its Lie bracket is not defined on the set of all local sections.\\
The two last sections show that, under adequate assumptions, the prolongation of a projective (resp. direct) limit of a projective sequence (resp. ascending sequence) of Banach Lie algebroids is exactly the prolongation of the projective (resp. direct) limit of this sequence. \\

In order to make the two last sections more affordable for non-informed readers, we have added two appendices recalling the needed concepts and results on projective and direct limits.

\section{Convenient Lie algebroid}
\label{_ConvenientLieAlgebroid}
\subsection{Local identifications and expressions in a convenient bundle}${}$\\
\label{___LocalIdentificationsAndExpressionsInAConvenientBundle}

\emph{In all this paper, we work in the convenient setting and we refer to \cite{KrMi}.}\\

Consider a convenient vector bundle $\pi: \mathcal{A}\to M$ where the typical fibre is a convenient linear space $\mathbb{A}$. For any open subset $U\subset M$,  we denote by $C^\infty(U)$ the ring of smooth functions on $U$ and  by $\Gamma\left(  \mathcal{A}_U \right) $ the $C^\infty(U)$-module of smooth sections of the restriction $ \mathcal{A}_U$ of $\mathcal{A}$ over $U$, simply  $\Gamma(\mathcal{A})$ when $U=M$.\\
Consider a chart  $(U,\phi)$ on $M$ such that we have a trivialization $\tau:{ \mathcal{A}}_U\to \psi(U)\times \mathbb{A} $
Then $T\phi$ is a trivialization of $TM_U$  on $\phi(U)\times \mathbb{M}$ and
 $T\tau$ is a trivialization of $T \mathcal{A}_U$ on $\phi(U)\times \mathbb{A}\times\mathbb{M}\times\mathbb{A}$.\\
For the sake of simplicity, we will denote these trivializations:
\begin{description}
\item[--]
	$ \mathcal{A}_U= \mathcal{A} _{|U}\equiv \phi(U)\times\mathbb{\mathbb{A} }$;
\item[--]
	$TM_U=TM_{|U}\equiv \phi(U)\times\mathbb{M}$;
\item[--]
	$T \mathcal{A}_U=T \mathcal{A}_{|{ \mathcal{A}_{(U)}}}\equiv(\phi(U)\times{\mathbb{A} })\times(\mathbb{M}\times{\mathbb{A} })$;
\item[--]
	$T \mathcal{A}^*_{|  \mathcal{A}^*_U}\equiv \phi(U)\times\mathbb{A}^{*}\times\mathbb{M}\times \mathbb{A}^*$.
\end{description}
where $U \subset M$ is identified with $\psi(U)$.\\

We will also use the following associated local coordinates where $\equiv$ stands for the representation in the corresponding trivialization.
\begin{description}
\item[]
	$\mathfrak{a}=(x,a)\equiv(\mathsf{x,a})\in U\times \mathbb{A}$
\item[]
	$(x,v)\equiv (\mathsf{x,v})\in U\times\mathbb{M}$
\item[]
	$(\mathfrak{a},\mathfrak{b})\equiv(\mathsf{x,a, v ,b})\in  U\times\mathbb{A}\times\mathbb{M}\times \mathbb{A}.$
\item[]
	$(\sigma,w \eta)\equiv (\mathsf{x,\xi,w_1,\eta_2})\in U\times\mathbb{A}^{*}\times\mathbb{M}\times \mathbb{A}^*$
\end{description}

\subsection{Convenient anchored bundle}
\label{___ConvenientAnchoredBundle}

Let $\pi:\mathcal{A}\to M$ be a convenient vector bundle whose fibre is a convenient linear space $\mathbb{A}$.
\begin{definition}
\label{D_Anchor}
A morphism of vector bundles $\rho:\mathcal{A}\to TM$ is called an \textit{anchor} and the quadruple $ \left( \mathcal{A},\pi,M,\rho \right) $ is called a convenient anchored bundle\index{convenient!anchored bundle}.
\end{definition}

\begin{notations}
\label{N_Anchor}
${}$
\begin{enumerate}
\item
In this section, if  there is no ambiguity, the anchored bundle  $(\mathcal{A},\pi,M,\rho)$ is fixed and, in all this work, the Lie bracket of vector fields on a convenient manifold will be simply denoted  $[.,.]$.
\item
For any open set $U$ in $M$, the morphism $\rho$ gives rise to a $C^\infty(U)$-morphism of modules ${\rho}_U:\Gamma\left(  \mathcal{A}_U\right) \to \mathfrak{X}(U) $ defined, for any $x \in M$ and any smooth section $\mathfrak{a}$ of $\mathcal{A}_U$, by:
\[
\left(  {\rho}_U\left(  \mathfrak{a} \right) \right)  \left(  x\right)  =\rho\left(  \mathfrak{a}\left(  x\right)  \right)
\]
and still denoted by $\rho$.
\item
For any convenient spaces $\mathbb{E}$ and $\mathbb{F}$, we denote by $\operatorname{L}(\mathbb{E},\mathbb{F})$, the convenient space of bounded linear operators from $\mathbb{E}$ to $\mathbb{F}$ and for $\mathbb{E}=\mathbb{F}$, we set   $\operatorname{L}(\mathbb{E}):=\operatorname{L}(\mathbb{E},\mathbb{F})$; $\operatorname{GL}(\mathbb{E})$ is the group of bounded automorphisms of $\mathbb{E}$.
\item
In local coordinates in a chart $(U,\phi)$, the restriction of $\rho$ to $U$ will give rise to a smooth field $\mathsf{x} \mapsto \mathsf{\rho_x}$ from $\phi(U)\equiv U$ to $\operatorname{L}(\mathbb{A},\mathbb{M})$.
\end{enumerate}
\end{notations}

\subsection{Almost Lie bracket}
\label{___AlmostLieBracket}
\begin{definition}
\label{D_AlmostLieBracketOnAnAnchoredBundle}
An almost Lie bracket on an anchored bundle $\mathcal{A}$ is a sheaf of skew-symmetric bilinear maps
\[
\lbrack.,.]_{\mathcal{A}_U}:\Gamma\left(  \mathcal{A}_U\right)  \times\Gamma\left(  \mathcal{A}_U\right)
\to\Gamma\left(  \mathcal{A}_U\right)
\]
for any open set $U\subseteq M$ which satisfies the following properties:
\begin{enumerate}
\item [\textbf{(AL 1)}]
The Leibniz identity:\index{Leibniz identity}
\[
\forall\left(  \mathfrak{a}_{1},\mathfrak{a}_{2}\right)  \in \Gamma \left( \mathcal{A}_U \right)  ^{2}, \forall f \in C^\infty(M)
,\ [\mathfrak{a}_{1},f\mathfrak{a}_{2}]_{\mathcal{A}}=f.[\mathfrak{a}_{1},\mathfrak{a}_{2}]_{\mathcal{A}}+df(\rho(\mathfrak{a}_{1})).\mathfrak{a}_{2}.
\]
\item[\textbf{(AL 2)}]
For any  open set $U\subseteq M$  the map
\[
(\mathfrak{a}_1,\mathfrak{a}_2)\mapsto [\mathfrak{a}_1,\mathfrak{a}_2]_{\mathcal{A}_U}
\]
only depends on the $1$-jets of $\mathfrak{a}_1$ and $\mathfrak{a}_2$ of sections of $\mathcal{A}_U$.
\end{enumerate}
By abuse of notation, such a sheaf of almost Lie brackets will be denoted $[.,.]_\mathcal{A}$.
\end{definition}

\begin{remark}
\label{R_LocalGlobal} 
In finite dimension, the bracket is defined on global sections and induces a Lie bracket on local sections which depend on the $1$-jets of sections. In the convenient setting (as in the Banach one), if $M$ is not smoothly regular, the  set  of restrictions to some open set $U$ of global sections of $\mathcal{A}$ could be different from $\Gamma(\mathcal{A}_U)$ but, unfortunately, we have no example of such a situation.  Thus, any bracket defined on the whole space $\Gamma(\mathcal{A})\times \Gamma(\mathcal{A})$ will not give rise to a bracket on local  sections of $\mathcal{A}$ and, even if it is true, the condition \emph{\textbf{(AL 2)}} will not be true in general.
\end{remark}

In the context of local trivializations ($\S$ \ref{___LocalIdentificationsAndExpressionsInAConvenientBundle}), if $\operatorname{L}^2_{\operatorname{alt}}(\mathbb{A};\mathbb{A}) $ is the convenient space of bounded skew-symmetric operators on $\mathbb{A}$ with values in $\mathbb{A}$, there exists a smooth field
\[
\begin{array}{cccc}
\mathsf{C}:    & U             & \rightarrow   &\operatorname{L}_{\operatorname{alt}}^2(\mathbb{A},\mathbb{A}) \\
	           & \mathsf{x}    & \mapsto       & \mathsf{C}_\mathsf{x}
\end{array}
\]
such that, for $\mathfrak{a}_1(x)\equiv(\mathsf{x},\mathsf{a_1(x)})$ and $\mathfrak{a}_2(x)\equiv(\mathsf{x,a_2(x)})$, we have:
\begin{eqnarray}
\label{eq_loctrivct}
[\mathfrak{a}_1,\mathfrak{a}_2]_U(x) 
\equiv (\mathsf{x,C_x(a_1(x),a_2(x)})+d \mathsf{a_2(\rho_x(a_1(x)))}-d \mathsf{a_1(\rho_x(a_2(x)))}).
\end{eqnarray}

\subsection{Almost Lie algebroid and Lie algebroid}
\label{___AlmostLieAlgebroidAndLieAlgebroid}
\begin{definition}
\label{D_AlmostLieAlgebroid}
The quintuple $ \left( \mathcal{A},\pi,M,\rho,[.,.]_{\mathcal{A}} \right) $ where $ \left( \mathcal{A},\pi,M,\rho \right) $ is an anchored bundle and $[.,.]_{\mathcal{A}}$ an almost Lie bracket is called a convenient almost Lie algebroid\index{convenient almost Lie algebroid}\index{almost Lie algebroid!convenient}.
\end{definition}

In this way, the \emph{Jacobiator}\index{Jacobiator}  is the $\mathbb{R}$-trilinear
map $J_{\mathcal{A}_U}:\Gamma(\mathcal{A}_{U})^{3}\to\Gamma(\mathcal{A}_{ U})$ defined, for any open set $U$ in $M$ and any section $\left(  \mathfrak{a}_{1}, \mathfrak{a}_{2}, \mathfrak{a}_{3} \right) \in \Gamma(\mathcal{A}_{U})^3$ by
\[
J_{\mathcal{A}_U}(\mathfrak{a}_{1},\mathfrak{a}_{2},\mathfrak{a}_{3}
)=[\mathfrak{a}_{1},[\mathfrak{a}_{2},\mathfrak{a}_{3}]_{\mathcal{A}}]_{\mathcal{A}}+[\mathfrak{a}_{2},[\mathfrak{a}_{3},\mathfrak{a}_{1}]_{\mathcal{A}}]_{\mathcal{A}}+[\mathfrak{a}_{3},[\mathfrak{a}_{1},\mathfrak{a}_{2}]_{\mathcal{A}}]_{\mathcal{A}}.
\]

\begin{definition}
\label{D_ConvenientLieAlgebroid}
A  convenient Lie algebroid \index{convenient Lie algebroid}\index{Lie algebroid!convenient} is a convenient almost Lie algebroid $ \left( \mathcal{A},\pi,M,\rho ,[.,.]_{\mathcal{A}} \right) $ such that the associated jacobiator $J_{\mathcal{A}_U}$ vanishes identically on each module $\Gamma(\mathcal{A}_{U})$ for all open sets  $U$  in $M$.
\end{definition}

We then have the following result (cf. \cite{BCP}, Chapter 3):
\begin{proposition}
\label{P_EquivalenceMorphismJEtensor} 
Consider an  almost convenient Lie algebroid  $\left( \mathcal{A},\pi,M,\rho ,[.,.]_{\mathcal{A}} \right) $.
\begin{enumerate}
\item
For any open set $U\subseteq M$  and for all $\left( \mathfrak{a}_{1},\mathfrak{a}_{2} \right) \in \Gamma\left( \mathcal{A}_{ U} \right) ^2$, the map
\[
\left( \mathfrak{a}_{1},\mathfrak{a}_{2} \right)\mapsto \rho \left( [\mathfrak{a}_1,\mathfrak{a}_2]_\mathcal{A} \right) -[\rho(\mathfrak{a}_1), \rho(\mathfrak{a}_2)]
\]
only depends on the $1$-jet of $\rho$ at any $x\in U$ and the values of $\mathfrak{a}_1$ and $\mathfrak{a}_2$ at $x$.
\item
If the  jacobiator $J_{\mathcal{A}_U}$ vanishes identically,  then  we have:
\begin{equation}
\label{eq_rhoCompatible}
\forall \left( \mathfrak{a}_{1},\mathfrak{a}_{2} \right) \in \Gamma\left( \mathcal{A}_{ U} \right) ^2,\; \rho \left( [\mathfrak{a}_1,\mathfrak{a}_2]_\mathcal{A} \right) =[\rho(\mathfrak{a}_1), \rho(\mathfrak{a}_2)].
\end{equation}
\item
If the property (\ref{eq_rhoCompatible}) is true, then $J_{\mathcal{A}_U}$  is a bounded trilinear $C^\infty(U)$-morphism from $ \Gamma\left( \mathcal{A}_U \right)^3$ to  $\Gamma\left( \mathcal{A}_U \right)$ which takes values in $\ker \rho$ over $U$.\\
\end{enumerate}
\end{proposition}

If for each open set $U$, the assumption (2) of Proposition \ref{P_EquivalenceMorphismJEtensor} is satisfied, (3) implies that the family $\{ J_{\mathcal{A}_U}, U\textrm{ open  in } M \}$ defines a sheaf of trilinear morphisms from the sheaf $\{ (\Gamma(A_U))^3, U \textrm{ open  in } M\}$ into the sheaf $\{ \Gamma(A_U), U \textrm{ open  in } M\}$. This sheaf will be denoted  $J_{\mathcal{A}}$.

\begin{corollary}
\label{C_rhoLieAlgebraMorphism}
If $ \left( \mathcal{A},\pi,M,\rho,[.,.]_{\mathcal{A}} \right) $ is a convenient Lie algebroid, then $\rho$ induces a morphism of Lie algebras from $\Gamma(\mathcal{A}_{ U})$ into $\mathfrak{X}(U)$ for any open sets $U$ in $M$.
\end{corollary}

\begin{definition}
\label{D_SplitConvenientLieAlgebroid}
A convenient Lie algebroid $\left( \mathcal{A},\pi,M,\rho,[.,.]_{\mathcal{A}} \right) $  will be called split\index{split convenient Lie algebroid} if, for each $x\in {M}$, the kernel of $\rho_x=\rho_{| \pi^{-1}(x)}$ is supplemented in $ \pi^{-1}(x)$.
\end{definition}
For example, if $ \operatorname{ker}\rho_x$ is finite dimensional or finite codimensional for all $x\in M$ or if $\mathbb{A}$ is a Hilbert space, then $(\mathcal{A},\pi,M,\rho,[.,.]_{\mathcal{A}})$ is split. Another particular situation is the case where the anchor $\rho\equiv 0$ and then  $\left( \mathcal{A},\pi,M,\rho,[.,.]_{\mathcal{A}} \right) $ is a \emph{Lie algebra Banach bundle}.

\subsection{Structure of partial Lie algebroid}
\label{___StructureofPartialLieAlgebroid}
We have the following generalization of the notion of convenient Lie algebroid:
\begin{definition}
\label{D_PartialConvenientLieAlgebroid}
Let $(\mathcal{A},\pi,M,\rho)$ be a convenient anchored bundle. Consider a sub-sheaf $\mathfrak{P}_M$ of the sheaf $\Gamma(\mathcal{A})_M$ of sections of $\mathcal{A}$.  Assume that  $\mathfrak{P}_M$ can be provided with a structure of Lie algebras sheaf  which satisfies, for any open set $U$ in $M$:
\begin{enumerate}
\item[(i)] 
for any $(\mathfrak{a}_1,\mathfrak{a}_2)\in \left( \mathfrak{P}(U) \right) ^2$ and any $f\in C^\infty(U)$, we have the Leibniz conditions
\begin{eqnarray}
\label{eq_rhoCompatibilitySheaf}
[\mathfrak{a}_1,f\mathfrak{a}_2]_{\mathfrak{P}(U)}=df(\rho(\mathfrak{a}_1))\mathfrak{a}_2+f[\mathfrak{a}_1,\mathfrak{a}_2]_{\mathfrak{P}(U)}
\end{eqnarray}
\item[(ii)]
the Lie bracket $[.,.]_{\mathfrak{P}(U)}$ on $\mathfrak{P}(U)$ only depends on the $1$-jets of sections of ${\mathfrak{P}(U)}$;
\item[(iii)]
$\rho$ induces  a Lie algebra morphism from  $\mathfrak{P}(U)$ to $\mathfrak{X}(U)$.
\end{enumerate}
Then $ \left( \mathcal{A},\pi, M,\rho,\mathfrak{P}_M \right)$ is called a convenient partial Lie algebroid\index{partial Lie algebroid}. The family $\{ [.,.]_{\mathfrak{P}(U)}, U \textrm{ open set in }M \}$ is called a sheaf bracket\index{sheaf bracket} and is denoted $[.,.]_\mathcal{A}$. \\
A partial convenient Lie algebroid  $(\mathcal{A},\pi, M,\rho,\mathfrak{P}_M)$  is called strong\index{partial Lie algebroid!strong} if for any $x\in M$, the stalk\index{stalk}
\[
\mathfrak{P}_x=\underrightarrow{\lim}\{ \mathfrak{P}(U),\;\; \varrho^U_V:V \to U,\;\; U,V \textrm{ open neighbourhoods of } x : U \supset V \}
\]
is equal to $\pi^{-1}(x)$ for any $x\in M$.
\end{definition}

Any convenient Lie algebroid is a partial Lie algebroid.\\
More generally, if $(\mathcal{A},\pi, M,\rho)$ is a convenient anchored bundle, any convenient subbundle $\mathcal{B}$ of $\mathcal{A}$ such that $ \left( \mathcal{B}, \pi_{\mathcal{B}}=\pi_{| \mathcal{B}},M, \rho_{\mathcal{B}}=\rho_{| \mathcal{B}}, [.,.]_{\mathcal{B}} \right) $ is a convenient Lie algebroid  provided with a structure of convenient partial Lie algebroid on $\mathcal{A}$ which is not strong in general.
Another type of example of convenient partial Lie algebroids will be described in the context of the prolongation of a convenient Lie algebroid in the next section. This convenient partial Lie algebroid  will be a strong partial Lie algebroid.

\begin{remark}
\label{R_PartialLieBracket} 
In local coordinates, the Lie bracket $[.,.]_{\mathfrak{P}(U)}$ can be written as in (\ref{eq_loctrivct}).
\end{remark}

\subsection{Derivative operators}
\label{___DerivativeOperators}
\subsubsection{Preliminaries}
\label{____Preliminaries}

If $U$ is a $c^\infty$-open subset of  a convenient space $\mathbb{E}$, the space $C^\infty(U,\mathbb{F})$ of smooth maps from $U$ to a convenient space $\mathbb{F}$ is a convenient space (cf. \cite{KrMi}, 3.7 and 3.11).\\
The space $L(\mathbb{E},\mathbb{F})$ of bounded linear maps from $\mathbb{E}$ to $\mathbb{F}$ endowed with the topology of uniform convergence on bounded subsets in $\mathbb{E}$ is a closed subspace of $C^\infty(\mathbb{E},\mathbb{F})$ and so is a convenient space.\\
More generally, the set $L^{k}_{\operatorname{alt}}(\mathbb{E},\mathbb{F})$  of all bounded $k$-linear alternating mappings  from $\mathbb{E}^k$ to $\mathbb{F}$ endowed with the topology of uniform convergence of bounded sets is a closed subspace of $C^\infty(\mathbb{E}^k,\mathbb{F})$ (cf. \cite{KrMi}, Corollary 5.13) and so  $L^{k}_{\operatorname{alt}}(\mathbb{E},\mathbb{F})$ is a convenient space.\\
On the other hand, if $\bigwedge^k(\mathbb{E})$ is the set of alternating $k$-tensors on $\mathbb{E}$ then
$L^{k}_{\operatorname{alt}}(\mathbb{E}):=L^{k}_{\operatorname{alt}}(\mathbb{E},\mathbb{R})$ is isomorphic as a locally convex  topological space to $L_{\operatorname{alt}}^k(\mathbb{E})$ (cf. \cite{KrMi}, Corollary 5.9) and so has a natural structure of convenient space.\\
Recall that bounded linear maps are smooth (cf. \cite{KrMi}, Corollary 5.5).\\

Let us consider a convenient vector bundle $\pi:\mathcal{A}\rightarrow M$ with typical fibre $\mathbb{A}$. We study the bundle
\[
\begin{array}
[c]{cccc}
\pi^k:  & L_{\operatorname{alt}}^k(\mathcal{A})=\displaystyle\bigcup_{x\in M}L_{\operatorname{alt}}^k(\mathcal{A}_x) & \to & M\\
        & (x,\omega)  	                                                                                             & \mapsto 	& x
\end{array}
\]
Using any atlas for the bundle structure of $\pi:\mathcal{A}\rightarrow M$, it is easy to prove that $\pi^k:L_{\operatorname{alt}}^k(\mathcal{A})\rightarrow M$ is a convenient vector bundle. The vector space of local sections of $L_{\operatorname{alt}}^k(\mathcal{A}_U)$ is denoted by $\bigwedge^{k}\Gamma^*(\mathcal{A}_U)$ and is called the set of $k$-exterior differential forms on $\mathcal{A}_U$.  We denote by $\bigwedge^k\Gamma^*(\mathcal{A})$ the sheaf of sections of  $\pi^k: L_{\operatorname{alt}}^k(\mathcal{A})\to M$  and $\bigwedge\Gamma^*(\mathcal{A})=\displaystyle\bigcup_{k=0}^\infty\bigwedge^k\Gamma^*(\mathcal{A})$ the sheaf of associated  graded exterior algebras.
\\

 {\bf In this section, we assume that $(\mathcal{A},\pi, M,\rho)$ is an anchored bundle and  that $ \left( \mathcal{A},\pi, M,\rho,\mathfrak{P}_M \right)$ is a fixed strong partial Lie algebroid}.\\ 

This situation is always satisfied if $(\mathcal{A},\pi, M,\rho, \left[.,.\right]_{\mathcal{A}})$ is a Lie algebroid and  occurs for the prolongation of a convenient Lie algebroid (cf $\S$  \ref{__ProlongationOfAConvenientLieAlgebroidAlongAFibration}). This context also occurs in the setting of partial Poisson manifolds (cf. \cite{PeCa}).

\subsubsection{ Insertion operator }
\label{____InteriorProduct}
Let $\mathfrak{a}$  be a local section of $\mathcal{A}$ defined on an open set $U$. As in \cite{KrMi}, 33.10, we have
\begin{proposition}
\label{D_InsertionOperator}
The insertion operator\index{insertion operator}\index{operator!insertion} $i_\mathfrak{a}$ is the  graded endomorphism of degree $-1$ defined by:
\begin{enumerate}
\item
\begin{enumerate}
\item[(i)]
For any function $f\in C^\infty(U)$
\begin{eqnarray}
\label{eq_i0} 
i_{\mathfrak{a}}\left(  f\right) =0
\end{eqnarray}
\item[(ii)]
For any $k$-form $\omega$ (where $k>0$),
\begin{eqnarray}
\label{eq_iq} 
\left( i_{\mathfrak{a}}\omega\right)  \left(  \mathfrak{a}_{1},\dots,\mathfrak{a}_{k}\right)(x)=\omega(
   \mathfrak{a}(x),\mathfrak{a}_1(x),\dots,\mathfrak{a}_k(x)).
\end{eqnarray}	
\end{enumerate}
\item
$\hfil{
i_\mathfrak{a}(\omega \wedge \omega^\prime)=i_\mathfrak{a}(\omega)\wedge \omega^\prime+(-1)^{{\rm deg}\omega}i_\mathfrak{a}(\omega^\prime).
}$
\end{enumerate}
\end{proposition}

\subsubsection{Lie derivative}
\label{____Liederivative}

\begin{proposition}
\label{P_Liekform} 
For $k\geq 0$,  let $\omega$ be a  local $k$-form that is an element of  $\bigwedge^{k}\Gamma^*(\mathcal{A}_U)$ for some  open set $U$ of $M$. Given any section $\overline{\mathfrak{a}}\in \mathfrak{P}(U)$,  the \emph{Lie derivative}\index{Lie derivative} with respect to $\overline{\mathfrak{a}}$ on sections of  $\mathfrak{P}(U)$, denoted by  $L_{\overline{\mathfrak{a}}}^\rho$,  is  the  graded endomorphism with degree $0$ defined  in the following way:
\begin{enumerate}
\item
For any function $f\in C^\infty(U)$,
\begin{eqnarray}
\label{eq_L0}
L_{\overline{\mathfrak{a}}}^{\rho}(f) = i_{{\rho}\circ\overline{ \mathfrak{a}}}\left(  df\right)
\end{eqnarray}
where $L_{X}$ denote the usual Lie derivative with respect to the vector field  $X$ on $M$.
\item
For any $k$-form $\omega$ (where $k>0$),
\begin{eqnarray}
\label{eq_Lq} 
\begin{aligned}
	\left(  L_{\overline{\mathfrak{a}}}^{\rho}\omega\right)  \left(  \mathfrak{a}_{1},\dots,\mathfrak{a}_{k}\right)(x)=
    & L_{\overline{\mathfrak{a}}}^{\rho}\left(  \omega\left(  \mathfrak{a}_{1},\dots,\mathfrak{a}_{k}\right)  \right)(x)\\
	&-{\displaystyle\sum\limits_{i=1}^{k}} \omega\left(  \mathfrak{a}_{1},\dots,\mathfrak{a}_{i-1}
	,\left[  \overline{\mathfrak{a}},\overline{\mathfrak{a}}_{i}\right]  _{\mathcal{A}},\mathfrak{a}_{i+1},\dots,\mathfrak{a}_{k}\right)(x)
\end{aligned}
\end{eqnarray}	
where $\overline{\mathfrak{a}}_i$ is any section of $\mathfrak{P}(U)$ such that $\overline{\mathfrak{a}}_i(x)={\mathfrak{a}}_i(x)$ for $i \in \{1, \dots, k\} $.
\end{enumerate}
\end{proposition}

\begin{proof}
Since the problem is local, we may assume that $U$ is a $c^\infty$-open set in $\mathbb{M}$ over which $\mathcal{A}$ is trivial. Fix some section $\mathfrak{a}$ of $\mathcal{A}$ on $U$ and fix some $x\in U$. After shrinking $U$ if necessary, if  $f$ is a smooth function on $U$, it is clear that (\ref{eq_L0}) is well defined.\\
For $k>0$, let $\omega\in \bigwedge^{k}\Gamma^*(\mathcal{A}_U)$. Since we have a strong partial Lie algebroid, this implies that for any $k$-uple $(\mathfrak{a}
_1,\dots,\mathfrak{a}_k)$ of local sections on $U$ of $\mathcal{A}$, the value $\omega(\mathfrak{a}_1(x),\dots,\mathfrak{a}_k(x))$ is well defined. Let $\overline{\mathfrak{a}}_i\in \mathfrak{P}(U)$ such that $\mathfrak{a}_i(x)=\overline{\mathfrak{a}}_i(x)$ for $i \in \{1,\dots,k\}$, we apply the formula (\ref{eq_Lq}) to $(\overline{\mathfrak{a}},\overline{\mathfrak{a}}_1,\dots, \overline{\mathfrak{a}}_k)$. In our context, $\omega$ is a smooth field over $U$ with values in $L_{\operatorname{alt}}^k(\mathbb{A})$ and each $\overline{\mathfrak{a}}_i$ is a smooth map from $U$ to $\mathbb{A}$. In this way, we have

\begin{align*}
L^\rho_{\overline{\mathfrak{a}}}\omega \left( \overline{ \mathfrak{a}}_{1},\dots,\overline{ \mathfrak{a}}_{k}\right)(x)
=& d_x\omega \left( \rho(\overline{\mathfrak{a}}); \overline{  \mathfrak{a}}_{1}(x),\dots,\overline{ \mathfrak{a}}_{k}(x) \right)\\ &+\sum_{i=1}^k\omega  \left( \overline{  \mathfrak{a}}_{1}(x),\dots, d_x\overline{ \mathfrak{a}}_{i}(\rho(\overline{\mathfrak{a}})), \dots,\overline{ \mathfrak{a}}_{k}(x) \right)
\end{align*}

Since $[.,.]_{\mathfrak{P}(U)}$ only depends on the $1$-jets of sections, as for an almost Lie bracket (cf. Remark \ref{R_PartialLieBracket}), we have:
\[
\left[  \overline{\mathfrak{a}},\overline{\mathfrak{a}}_{i}\right]_{\mathcal{A}}(x)=d_x\overline{\mathfrak{a}}_{i}(\rho( \overline{\mathfrak{a}}))-d_x\overline{\mathfrak{a}}(\rho( \overline{\mathfrak{a}}_i)+ C_x( \overline{\mathfrak{a}}(x), \overline{\mathfrak{a}}_i(x)).
\]
It follows that we have
\begin{align*}
\left(  L_{\overline{\mathfrak{a}}}^{\rho}\omega\right)  \left(  \mathfrak{a}_{1},\dots,\mathfrak{a}_{k}\right)(x)=  d_x\omega\left(\rho(\overline{\mathfrak{a}}); \overline{  \mathfrak{a}}_{1}(x),\dots,\overline{ \mathfrak{a}}_{k}(x)\right))\\
+ {\displaystyle\sum\limits_{i=1}^{k}} \omega\left(  \overline{\mathfrak{a}}_{1}(x),\dots,\overline{\mathfrak{a}}_{i-1}(x)
	,d_x\overline{\mathfrak{a}}(\rho( \overline{\mathfrak{a}}_i)+ C_x( \overline{\mathfrak{a}}(x), \overline{\mathfrak{a}}_i(x)),\overline{\mathfrak{a}}_{i+1}(x),\dots,\overline{\mathfrak{a}}_{k}(x) \right) .
\end{align*}
which implies that $L_{\overline{\mathfrak{a}}}^{\rho}\omega$ is a well defined $k$-skew symmetric form on  $\mathcal{A}_x$ on $U$ since its value only depends on the $1$-jet of $\overline{\mathfrak{a}}$ and of $\omega$ and the values of $(\mathfrak{a}_1,\dots,\mathfrak{a}_k)$ at $x$. Now, as $x\mapsto C_x$ is a smooth map from $U$ to $L_{\operatorname{alt}}^2(\mathbb{A})$ and so is bounded, the differential of  functions is a bounded morphism of convenient space (cf. \cite{KrMi}, 3), and $\rho$ is a bounded morphism of convenient spaces, this completes the proof according to the uniform boundedness principle given in  \cite{KrMi}, Proposition 30.1.
\end{proof}

\begin{remark}
\label{R_Liederivativef}
From the relation (\ref{eq_L0}), it is clear that the Lie derivative of a function is defined for any section of $\mathcal{A}_U$.  Of course, this is also true for any  $k$-form on a Lie algebroid.  But, for a strong partial Lie algebroid, this is not true for any $k$-form with $k>0$,  since the last formula in the previous proof shows clearly that $L_{\overline{\mathfrak{a}}}^{\rho}\omega$ also depends on  the $1$-jet of ${\overline{\mathfrak{a}}}$.
\end{remark}

\begin{remark}
\label{R_AlmostLieDerivative}   
Assume that $(\mathcal{A},\pi, M,\rho)$ is provided with an almost Lie bracket $\left[.,.\right]_{\mathcal{A}}$. Then the Lie derivative $ L_{{\mathfrak{a}}}^{\rho}\omega$  is again well defined by an evident adaptation of formula (\ref{eq_Lq} ) for any local section $\mathfrak{a}$ and $k$-form $\omega$ defined on some open set $U$. 
 Moreover, if  the Lie bracket on $ \left( \mathcal{A},\pi, M,\rho,\mathfrak{P}_M \right)$ is induced by the almost Lie bracket $\left[.,.\right]_{\mathcal{A}}$,  then the Lie derivative defined in Proposition \ref{P_Liekform} and the previous global one are compatible. 
\end{remark}

\subsubsection{Exterior derivative}
\label{____ExteriorDerivative}
At first, for any function $f$, we can also define the $1$-form $d_{\rho}f,$ by
\begin{eqnarray}
\label{eq_d0} 
d_{\rho}f={{\rho}^t}\circ df
\end{eqnarray}
where ${\rho}^t:T^{\prime}M\rightarrow \mathcal{A}^{\prime}$  is the transposed mapping of $ \rho$.

\smallskip

The Lie derivative with respect to any local section  $\mathfrak{a}$ of $\mathcal{A}$  commutes with $d_{\rho}$.

\medskip

The \emph{exterior differential}\index{exterior differential} on $\bigwedge\Gamma^*(\mathcal{A})$ is defined as follows:

\begin{proposition}
\label{P_Exreriordidderential}${}$
\begin{enumerate}
\item[(1)]
The exterior differential  $d_\rho$ is  the  graded endomorphism of  degree $1$   on $\bigwedge\Gamma^*(\mathcal{A})$ defined in the following way:
\begin{enumerate}
\item
    For any function $f$, $d_{\rho}f$ is defined previously;
\item
    For  $k>0$ and  any $k$-form  $\omega$, the exterior differential $d_{\rho}\omega$ is the unique $(k+1)$-form such that, for all $\mathfrak{a}_{0},\dots,\mathfrak{a}_{k}\in \Gamma(\mathcal{A})$,
\begin{eqnarray}
\label{eq_dext}
\begin{aligned}
	\left(  d_{\rho}\omega\right)  \left(  \mathfrak{a}_{0},\dots, \mathfrak{a}_{q}\right)(x)  &
	={\displaystyle\sum\limits_{i=0}^{q}}\left(  -1\right)  ^{i}L_{\mathfrak{a}_{i}}^{\rho
	}\left(  \omega\left(   \mathfrak{a}_{0},\dots,\widehat{ \mathfrak{a}_{i}},\dots, \mathfrak{a}_{q}\right)(x)
	\right) \\
	& +{\displaystyle\sum\limits_{0\leq i<j\leq q}}\left(  -1\right)
	^{i+j}\left(  \omega\left(  \left[   \overline{\mathfrak{a}}_{i}, \overline{\mathfrak{a}}_{j}\right]  _{\mathcal{A}}, \mathfrak{a}_{0}
	,\dots,\widehat{ \mathfrak{a}_{i}},\dots,\widehat{ \mathfrak{a}_{j}},\dots, \mathfrak{a}_{q}\right)  \right)(x)
\end{aligned}
\end{eqnarray}
where $\overline{\mathfrak{a}}_i$ is any section of $\mathfrak{P}(U)$ such that $\overline{\mathfrak{a}}_i(x)={\mathfrak{a}}_i(x)$ for $i \in \{0,\dots k\}$.
\end{enumerate}
\item[(2)] 
For any  $k$-form $\eta$, $l$-form $\zeta$ where $\left( k,l \right) $ in $\mathbb{N}^2$, we then have the following property  
\begin{equation}
\label{Eq_WedgeProduct}
d_\rho(\eta\wedge\zeta)=d_\rho(\eta)\wedge \zeta+(-1)^k\eta\wedge d_\rho(\zeta).
\end{equation}
\begin{equation}
\label{Eq_dcircd}
 d_{\rho}\circ d_{\rho}={d_\rho}^2=0.
\end{equation}
\end{enumerate}
\end{proposition}
\bigskip
\begin{proof}${}$\\
(1) Using the same context as in the proof of Proposition \ref{P_Liekform}, on the one hand, in local coordinates, we have
\begin{align*}
L_{\overline{\mathfrak{a}}_{i}}^{\rho} \left( \omega \left( \mathfrak{a}_{0},\dots,\widehat{\mathfrak{a}_{i}},\dots, \mathfrak{a}_{q}\right)(x) \right) 
&= d_x\omega\left(\rho({\mathfrak{a}}_i); \overline{  \mathfrak{a}}_{1}(x),\dots,\widehat{\overline{ \mathfrak{a}}_{i}},\dots,\overline{ \mathfrak{a}}_{k}(x)\right)\\
&\;\;\;+\sum_{i=1}^k\omega  \left(\overline{  \mathfrak{a}}_{1}(x),\dots, d_x\overline{ \mathfrak{a}}_{j}(\rho({\mathfrak{a}}_i)), \dots,\overline{ \mathfrak{a}}_{k}(x)\right) .\\
\end{align*}
On the other hand, we have
\begin{eqnarray*}
\begin{aligned}
&\omega\left(  \left[   \overline{\mathfrak{a}}_{i}, \overline{\mathfrak{a}}_{j}\right]  _{\mathcal{A}}, \mathfrak{a}_{0}
	,\dots,\widehat{ \mathfrak{a}_{i}},\dots,\widehat{ \mathfrak{a}_{j}},\dots, \mathfrak{a}_{q}\right) (x)\\
&=\omega\left( d_x\overline{\mathfrak{a}}_{i}(\rho( \overline{\mathfrak{a}}_j))-d_x\overline{\mathfrak{a}}_j(\rho( \overline{\mathfrak{a}}_i)+ C_x( \overline{\mathfrak{a}}_i(x), \overline{\mathfrak{a}}_j(x)), \overline{\mathfrak{a}}_i(x)), \overline{\mathfrak{a}}_{j}, \mathfrak{a}_{0}
	,\dots,\widehat{ \mathfrak{a}_{i}},\dots,\widehat{ \mathfrak{a}_{j}},\dots, \mathfrak{a}_{k}\right)(x).
\end{aligned}
\end{eqnarray*}
Finally, as $\rho(\mathfrak{a}_i(x))=\rho(\overline{\mathfrak{a}}_i(x))$ for $i \in \{0,\dots, k\}$, we obtain:
\begin{eqnarray*}
\begin{aligned}
&\left(  d_{\rho}\omega\right)  \left(  \mathfrak{a}_{0},\dots, \mathfrak{a}_{q}\right)(x)={\displaystyle\sum\limits_{i=0}^{k}}\left(  -1\right)  ^{i}  d_x\omega\left(\rho(\overline{\mathfrak{a}}_i); \overline{  \mathfrak{a}}_{1}(x),\dots,\widehat{\overline{ \mathfrak{a}}_{i}},\dots,\overline{ \mathfrak{a}}_{k}(x) \right) \\
&+{\displaystyle\sum\limits_{0\leq i<j\leq k}}\left(  -1\right)^{i+j}\left(  \omega\left(  C_x\left( \overline{\mathfrak{a}}_i(x), \overline{\mathfrak{a}}_j(x)\right), \overline{\mathfrak{a}}_{0}(x),\dots,\widehat{\overline{ \mathfrak{a}}_{i}}(x),\dots,\widehat{ \overline{\mathfrak{a}}_{j}}(x),\dots, \overline{\mathfrak{a}}_{k}(x) \right)\right)
\end{aligned}
\end{eqnarray*}
Since this value only depends on the $1$-jet of $\omega$ at $x$ and the value of each  $\overline{\mathfrak{a}}_i(x)$ for $i \in \{0,\dots k\}$, it follows that $d_{\rho}\omega$  is a well defined $(k+1)$-form by the same arguments as at the end of the proof of Proposition \ref{P_Liekform} .\\

(2) According to the definition of the wedge product, the last formula in local coordinates for $ \left(  d_{\rho}\omega\right) \left(  \mathfrak{a}_{0},\dots, \mathfrak{a}_{q}\right)$ clearly implies
relation (\ref{Eq_WedgeProduct}).\\
Since the Lie bracket on $\mathfrak{P}(U)$ satisfies the Jacobi identity for any open set $U$, and since the differential $d_\rho\omega$ only depends on the $1$-jet of $\omega$, as in finite dimension, it follows that $d_\rho(d_\rho \omega)=0$.
\end{proof}

\subsubsection{Nijenhuis endomorphism}
\label{____NijenhuisEndomorphism}

In this subsection, we only consider the case of a convenient Lie algebroid  $\left( \mathcal{A},\pi,M,\rho,[.,.]_{\mathcal{A}} \right) $.

Let $A$ be an endomorphism of $\mathcal{A}$. The \emph{Lie derivative of $A$} with respect to a local section $\mathfrak{a}$ is defined by
\begin{eqnarray}
\label{eq_LieDerivativeEndomorphism}
L^\rho_\mathfrak{a}A(\mathfrak{b})=[\mathfrak{a},A \left( \mathfrak{b} \right) ]_\mathcal{A}-A \left( [\mathfrak{a},\mathfrak{b}]_\mathcal{A} \right)
\end{eqnarray}
for all local or global sections $\mathfrak{b}$  with the same domain as $\mathfrak{a}$.\\

The \emph{Nijenhuis tensor}\index{Nijenhuis tensor}\index{tensor!Nijenhuis} of $A$ is the tensor of type $(2,0)$ defined by:
\begin{eqnarray}
\label{eq_NijenhuisEndomorphism}
N_A(\mathfrak{a},\mathfrak{b})=[A\mathfrak{a},A\mathfrak{b}]_\mathcal{A}-A[A\mathfrak{a},\mathfrak{b}]_\mathcal{A}-A[\mathfrak{a},A\mathfrak{b}]_\mathcal{A}-[\mathfrak{a},\mathfrak{b}]_\mathcal{A}
\end{eqnarray}
for all local or global sections $\mathfrak{a}$ and $\mathfrak{b}$ with same domain.

\begin{remark}
\label{R_PartialAlgebroid} 
Consider a partial Lie algebroid  $ \left( \mathcal{A},\pi, M,\rho,\mathfrak{P}_M \right) $. If $A$ is an endomorphism of sheaves of $\mathfrak{P}_M$, the same formulae are well defined for any local section $\overline{\mathfrak{a}}$ of $\mathfrak{P}_M$. In this way, we can also define the Nijenhuis tensor $N_A$ as an endomorphism of sheaves of  $\mathfrak{P}_M$.
\end{remark}

\subsection{Lie morphisms and Lie algebroid morphisms}
\label{___LieAlgebroidMorphism}
Let $(\mathcal{A}_1, \pi_1, M_1,\rho_1, [.,.]_{\mathcal{A}_1})$ and  $(\mathcal{A}_2, \pi_2, M_2,\rho_2, [.,.]_{\mathcal{A}_2})$  be two convenient Lie algebroids.\\
We consider a bundle morphism $\Psi:\mathcal{A}_1 \to \mathcal{A}_2$ over $\psi:M_1 \to M_2$. \\

In the one hand, according to \cite{HiMa} in finite dimension, we can introduce:
\begin{definition}
\label{D_psiRelatedSections}
Consider  a section $\mathfrak{a}_1$ of $\mathcal{A}_1$ over an open set $U_1$ and a section $\mathfrak{a}_2$ of $\mathcal{A}_2$ over  an open set $U_2$ which contains $\psi(U_1)$. We say that the pair of sections $(\mathfrak{a}_1,\mathfrak{a}_2)$ are $\psi$-related\index{related pairs of sections} if we have
\begin{description}
\item{\bf (RS)}
$\Psi\circ \mathfrak{a}_1=\mathfrak{a}_2\circ \psi$.
\end{description}
\end{definition}

\begin{definition}
\label{D_LieMorphism}
$\Psi$ is called a Lie morphism\index{Lie morphism}\index{morphism!Lie}  over $\psi$ if it fulfills the following conditions:
\begin{description}
\item[\textbf{(LM 1)}]
$\rho_2 \circ \Psi= T\psi \circ \rho_1$;
\item[\textbf{(LM 2)}]
$\Psi \circ [\mathfrak{a}_1,\mathfrak{a}^\prime_1]_{\mathcal{A}_1} = [\mathfrak{a}_2,\mathfrak{a}'_2]_{\mathcal{A}_2}\circ \psi$
for all $\psi$-related pairs of sections  $(\mathfrak{a}_1,\mathfrak{a}_2)$ and $(\mathfrak{a}^\prime_1,\mathfrak{a}^\prime_2)$.
\end{description}
\end{definition}

\begin{remark} 
\label{R_LieMorphismPartialAlgebroid} 
For $i \in \{1,2\}$,  let $ \left( \mathcal{A}_i, \pi_i, M_i,\rho_i, \mathfrak{P}_{M_i} \right) $ be a partial Lie algebroid.\\
We consider a sheaf  morphism $\Psi:\mathfrak{P}_{M_1} \to \mathfrak{P}_{M_2}$ over a smooth map $\psi:M_1 \to M_2$. Then the Definition \ref{D_psiRelatedSections}  makes sense for pairs of sections $(\mathfrak{a}_1,\mathfrak{a}_2)\in \mathfrak{P}_{M_1}\times\mathfrak{P}_{M_2}$ which are then called $\psi$-related. If $[.,.]_{\mathcal{A}_i}$ is the sheaf of Lie bracket defined on $\mathfrak{P}_{M_i}$ the assumption  \emph{\textbf{(LM 2)}} in Definition \ref{D_LieMorphism} also makes sense for two pairs of such $\psi$-related sections.\\
Thus $\Psi$ will be called a {\bf Lie morphism of partial convenient Lie algebroids} if it satisfies the assumptions \emph{\textbf{(LM 1)}}, and  \emph{\textbf{(LM 2)}}  for two pairs of $\psi$-related sections of $\mathfrak{P}_{M_1}\times\mathfrak{P}_{M_2}$.\\
\end{remark}

For any local  $k$-form $\omega$ on $\mathcal{A}_2$ defined on $U_2$, we denote by $\Psi^*\omega$ the local $k$-form on $\mathcal{A}_1$ defined on $U_1=\psi^{-1}(U_2)$ by:
\begin{eqnarray}
\label{eq_PullbackOmega}
(\Psi^*\omega)_{x_1}(\mathfrak{a}_1\dots\mathfrak{a}_k)=\omega_{\psi(x_1)}\left( \Psi(\mathfrak{a}_1),\cdots,\Psi(\mathfrak{a}_k) \right)
\end{eqnarray}
for all $x_1\in U_1$.\\

On the other hand, as classically in finite dimension, we can introduce:

\begin{definition}
\label{D_ClassicLieAlgebroidMorphism}
 $\Psi$ is a Lie algebroid morphism over $\psi$ if and only if we have
\begin{description}
\item[\textbf{(LAM 1)}]
$\Psi^*(d_{\rho _2}f) = d_{\rho _1} \left( f \circ \psi \right) $ for all $f \in C^\infty (U_2)$;
\item[\textbf{(LAM 2)}]
$\Psi^*(d_{\rho _2} \omega)=d_{\rho _1} \Psi^*(\omega)$  for any $1$-form $\omega$ on $ \{ \mathcal{A}_2 \} _{U_2}$.
\end{description}
\end{definition}

It is easy to see that condition \textbf{(LM 1)} and \textbf{(LAM 1)} are equivalent (cf. Proof of Proposition \ref{P_psiDiffeomorphism}).  Property \textbf{(LM 2)} implies property \textbf{(LAM 2)} for $\psi $-related sections but, in general, a pair of local sections $(\mathfrak{a}_1,\mathfrak{a}_2)$ of $\mathcal{A}_1$ and   $ \mathcal{A}_2$ are not $\psi$-related  while each member (\ref{eq_PullbackOmega}), for  any  two such pairs, is well defined. On the other hand, under the  assumption of  \textbf{(LM 2)},  we have 
\begin{eqnarray}
\label{eq_Psia1a2}
[\Psi(\mathfrak{a}_1),\Psi(\mathfrak{a}_2)]_{\mathcal{A}_2}\left(\psi(x_1)\right)=([\mathfrak{a}'_1,\mathfrak{a}'_2]_{\mathcal{A}_2})\left(\psi(x_1) \right) .
\end{eqnarray}
For any $x_1\in U_1$. Therefore  the relation \textbf{(LAM 1)} is satisfied  for any such pair  $ (\mathfrak{a}_1,\mathfrak{a}_2)$   of sections of $\mathcal{A}_1$ which are  $\psi$-related to a pair $ (\mathfrak{a}_1^\prime,\mathfrak{a}_2^\prime)$   of sections of $\mathcal{A}_2$. Of course, this property is no longer  true for any pair $(\mathfrak{a}_1,\mathfrak{a}_2)$ of local sections  of  $\mathcal{A}_2$ and so the bracket  "$[\Psi (\mathfrak{a}_1),\Psi (\mathfrak{a}^\prime_1)]_{\mathcal{A}_2}(\psi(x_1))$" is not defined. 
Thus, in general, both definitions are not comparable. However, if  $\psi$ is a local  diffeomorphism, we have:

\begin{proposition}
\label{P_psiDiffeomorphism}
Let $ \left( \mathcal{A}_1, \pi_1, M_1,\rho_1, [.,.]_{\mathcal{A}_1} \right) $ and  $ \left( \mathcal{A}_2, \pi_2, M_2,\rho_2, [.,.]_{\mathcal{A}_2} \right) $  be two convenient Lie algebroids.  We consider a bundle morphism $\Psi:\mathcal{A}_1\to \mathcal{A}_2$ over a local diffeomorphisms $\psi:M_1\to M_2$.
Then $\Psi$ is a Lie algebroid morphism if and only if it is a Lie morphism.
\end{proposition}

For instance, given any convenient Lie algebroid $ \left( \mathcal{A}, \pi, M,\rho, [.,.]_{\mathcal{A}} \right) $, then $\rho$ is Lie morphism and a Lie algebroid morphism from $ \left( \mathcal{A}, \pi, M,\rho, [.,.]_{\mathcal{A}} \right) $ to the convenient Lie algebroid $(TM, p_M, M, Id, [.,.])$.\\

\begin{proof}
Since the set  of differentials $\{df, \; f \textrm{ smooth map around } x_2\in M_2\}$
is a separating family on $T_{x_2}M_2$, by an elementary calculation, we obtain  the equivalence $\textbf{(LM 1)}\;\Leftrightarrow\; \textbf{(LAM 1)}$.\\
At first note that  \textbf{(LM 2)} and \textbf{(LAM 2)} are  properties of germs. Thus the equivalence is in fact a local problem.
Fix some $x^0_1\in M_1$ and $x^0_2=\psi(x^0_1)$. Since $\psi$ is a local diffeomorphism, the model of $M_1$ and of  $M_2$ are the same convenient space $\mathbb{M}$ and we have charts $(U_1,\phi_1)$ and $(U_2,\phi_2)$ around $x^0_1$ in $M_1$ and $x^0_2$ in $M_2$ such that  for $i \in \{1,2\}$:
\begin{description}
\item[--] 
$\phi_i(x_i)=0\in \mathbb{M}$; 
\item[--]
$\phi_2\circ \psi\circ \phi_1^{-1}$ is a diffeomorphism  between the $c^\infty$-open sets $\mathsf{U}_1:=\phi_1(U_1)$ and $\mathsf{U}_2:=\phi_2(U_2)$;
\item[--]
 a trivialization $\tau_i \{\mathcal{A}_i\} _{U_1}=U_i \times\mathbb{A}_i$. 
\end{description}

 Thus, without loss of generality, we may assume that $M_i$ is a $c^\infty$-open neighbourhood of $0\in \mathbb{M}$, $\psi$ is a diffeomorphism from $M_1$ to $M_2$ and  $\mathcal{A}_i=M_i\times \mathbb{A}_i$. By the way, the anchor $\rho_i$ is a smooth map from $M_i$ to $\operatorname{L}(\mathbb{M},\mathbb{A}_i)$ and each section $\mathfrak{a}_i$ of $\mathcal{A}_i$ is a smooth map from $M_i$ to $\mathbb{A}_i$. Under this context,  on the one hand, for all $x_1\in M_1$,  we have
\begin{eqnarray*}
\begin{aligned}
\Psi^*d_{\rho_2}\omega(\mathfrak{a},\mathfrak{a}')(x_1) & = {d}_{\psi(x_1)} \left( \omega(\Psi(\mathfrak{a}') \right) \left( \rho_2\circ \Psi(\mathfrak{a}) \right)) -{d}_{\psi(x_1)} \left( \omega(\Psi(\mathfrak{a}) \right) \left(\rho_2\circ \Psi(\mathfrak{a}') \right) \\
    &-\omega \left( [\Psi(\mathfrak{a}),\Psi(\mathfrak{a}')]_{\mathcal{A}_2} \right) \left( \psi(x_1) \right).\\
\end{aligned}
\end{eqnarray*}
On the other hand, we have
\begin{eqnarray*}
\begin{aligned}
d_{\rho_1}\Psi^*\omega(\mathfrak{a},\mathfrak{a}')(x_1)&
={d}_{\psi(x_1)}\omega \left( (\Psi(\mathfrak{a}') \right) \left( T\psi\circ\rho(\mathfrak{a}) \right)\\
&-{d}_{\psi(x_1)}\omega \left( (\Psi(\mathfrak{a}) \right) \left(T\psi\circ\rho(\mathfrak{a}') \right)
-\omega \left( \Psi([\mathfrak{a},\mathfrak{a}']_{\mathcal{A}_1} \right) (\psi(x_1)).\\
\end{aligned}
\end{eqnarray*}

Note that for any two pairs of $\psi$-related sections $(\mathfrak{a}_1,\mathfrak{a}_1')$ and $(\mathfrak{a}_2,\mathfrak{a}_2')$, as $\psi$ is a diffeomorphism,  \textbf{(LM 2)} is equivalent to
\begin{eqnarray}\label{eq_Psia1a2}
[\Psi(\mathfrak{a}_1),\Psi(\mathfrak{a}_2)]_{\mathcal{A}'}\left(\psi(x_1)\right)=\Psi([\mathfrak{a}_1,\mathfrak{a}_2]_{\mathcal{A}})\left(\psi(x_1)\right)
\end{eqnarray}
for all $x_1\in M_1$. 

Thus, if  \textbf{(LM 1)}  and \textbf{(LM 2)} are true, then, in the previous local context, \textbf{(LAM 2)} is equivalent to
\begin{eqnarray}
\label{eq_OmegaBracket}
\omega \left( \Psi([\mathfrak{a}_1,\mathfrak{a}_2]_{\mathcal{A}_1} \right)(\psi(x_1)) =\omega\left( [\Psi(\mathfrak{a}_1),\Psi(\mathfrak{a}_2)]_{\mathcal{A}'} \right) (\psi(x_1))
\end{eqnarray}
for any $1$-form $\omega$ on $\mathcal{A}_2$  and any $x_1\in M_1$. As $\psi$ is a diffeomorphism for any pair of sections $(\mathfrak{a}_1, \mathfrak{a}_2)$ of $\mathcal{A}_1$ if we set $  \mathfrak{a}
_1^\prime=\Psi(\mathfrak{a}_1)\circ\psi^{-1}$ and $  \mathfrak{a}_2^\prime=\Psi(\mathfrak{a}_2)\circ\psi^{-1}$, then $(\mathfrak{a}_1, \mathfrak{a}^\prime_1)$ and  $(\mathfrak{a}_2, \mathfrak{a}^\prime_2)$  are $\psi$-related and  it follows that \textbf{(LM 1)} and \textbf{(LM 2)} implies \textbf{(LAM 1)} and \textbf{(LAM 2)}.

Conversely, assume that  \textbf{(LAM 1)} and \textbf{(LAM 2)} are true.  Consider any
 two pairs of $\psi$-related sections $(\mathfrak{a}_1,\mathfrak{a}_1')$  and $(\mathfrak{a}_2,\mathfrak{a}_2')$. In this case  the relation (\ref{eq_OmegaBracket}) 
 evaluated  on $(\mathfrak{a}_1,\mathfrak{a}_2)$ is equivalent to \textbf{(LAM 2)} for any $1$-form $\omega $ on $U_2$. Since  around each point  in $M_2$, the set of germs of $1$-forms on $ \mathcal{A}_2$ is separating for germs of sections of $\mathcal{A}_2$ and as $\psi$ is a diffeomorphism, this implies (\ref{eq_OmegaBracket}).\\
It follows that the relation \textbf{(LM 2}) evaluated on both pairs   $(\mathfrak{a}_1,\mathfrak{a}_1')$  and $(\mathfrak{a}_2,\mathfrak{a}_2')$ is satisfied,  which ends the proof.
\end{proof}

\subsection{Foliations and  Banach-Lie algebroids}
\label{__FoliationsAndBanachLieAlgebroids}	
We first recall the classical notion of integrability of a distribution on a Banach manifold (cf. \cite{Pe1}).
		
Let $M$ be a Banach manifold.
\begin{enumerate}
\item
A distribution\index{distribution} $\Delta$ on $M$ is an assignment $\Delta: x\mapsto\Delta_{x}\subset T_{x}M$ on $M$ where $\Delta_{x}$ is a subspace of $T_{x}M$.  The distribution $\Delta$ is called closed if $\Delta_x$ is closed in $T_xM$  for all $x\in M$.		
\item
A vector field $X$ on $M$, defined on an open set Dom$(X)$, is called
tangent to a distribution $\Delta$ if $X(x)$ belongs to $\Delta_{x}$ for all
$x\in$Dom$(X)$. 			
\item
Let $X$ be a vector field tangent to a distribution $\Delta$ and $\operatorname{Fl}^X_t$ its flow. We say that $\Delta$ is $X$-invariant  if
$T_x\operatorname{Fl}^X_t(\Delta_x)=\Delta_{\operatorname{Fl}^X_t(x)}$
for all $t$ for which $\operatorname{Fl}^X_t(x)$ is defined.  			
\item
A distribution $\Delta$ on $M$ is called integrable if, for all
$x_{0}\in M$, there exists a weak submanifold $(N,\phi)$ of $M$ such that
$\phi(y_{0})=x_{0}$ for some $y_{0}\in N$ and $T\phi(T_{y}N)=\Delta_{\phi(y)}$
for all $y\in N$. In this case $(N,\phi)$ is called an integral manifold of
$\Delta$ through $x$. A leaf $L$ is a weak submanifold which is a maximal integral manifold.		
\item
A distribution $\Delta$ is called involutive if for any vector fields $X$
and $Y$ on $M$ tangent to $\Delta$ the Lie bracket $[X,Y]$ defined on
Dom$(X)\cap$Dom$(Y)$ is tangent to $\Delta$.\\	
\end{enumerate}

Classically, in the Banach context, when $\Delta$ is a supplemented subbundle of
$TM$, according to the Frobenius Theorem, involutivity implies integrability.\\	
In finite dimension, the famous results of H. Sussman and P. Stefan give
necessary and sufficient conditions for the integrability of smooth distributions.\\
Few generalizations of these results in the framework of Banach manifolds can be found in \cite{Ste}.

In  the context of this section, we have (cf. \cite{Pe1}):		
\begin{theorem}
\label{T_IntegrabilityDistributionRangeAnchor}
Let $(\mathcal{A},\pi,M,\rho,[.,.]_{\mathcal{A}})$ be a split Banach-Lie algebroid.\\
If $\rho(\mathcal{A})$ is a closed distribution, then this distribution is integrable.	
\end{theorem}
	
Note that if $\rho$ is a Fredholm morphism, the assumptions of Theorem \ref{T_IntegrabilityDistributionRangeAnchor} are always satisfied. In the Hilbert framework, only the closeness of $\rho$ is required.

\section{Prolongation of a convenient  Lie algebroid along a fibration}
\label{__ProlongationOfAConvenientLieAlgebroidAlongAFibration}
	
\subsection{Prolongation of  an anchored convenient  bundle }
\label{___ProlongationOfAnAnchoredConvenientBundle}
	
Let  ${\bf p}:\mathcal{E}\rightarrow M$ be a convenient vector bundle with typical fibre $\mathbb{E}$ and  $\mathcal{M}$ an open submanifold such that the restriction of ${\bf p}$ to $\mathcal{M}$ is a surjective fibration over $M$ of typical fibre $\mathbb{O}$ (open subset of $\mathbb{E}$). We consider some anchored convenient bundle $({\mathcal{A}},\pi,M,\rho)$.
		
\begin{notations}
\label{N_locM}
If  $(U,\phi)$ is a chart  such that $\mathcal{E}_U$ and $\mathcal{A}_U$ are trivializable, then $TM_U=TM_{| U}$ and $T\mathcal{M}_U$ are also trivializable. In this case, we have trivializations and  local coordinates
\begin{description}
\item[--]
$\mathcal{E}_U\equiv U\times \mathbb{E}$ and $\mathcal{M}_U\equiv U\times \mathbb{O}$  with local coordinates $\mathsf{m}=(\mathsf{x,e})$
\item[--]
$T\mathcal{M}_U\equiv (U\times \mathbb{O})\times\mathbb{M}\times \mathbb{E}$ with local coordinates $(\mathsf{m, v, z})$.
\end{description}
\end{notations}
For  $m\in {\bf p}^{-1}(x)$ we set
\[
{\mathbf{T}}^{\mathcal{A}}_m{\mathcal{M}}=\{(a,\mu)\in{\mathcal{A}}_x\times T_m{\mathcal{M}}: \; \rho(a)=T{\bf p}(\mu)\}.
\]
An element of ${\mathbf{T}}^{\mathcal{A}}_m{\mathcal{M}}$ will be denoted $(m,a,\mu)$.\\
We set ${\mathbf{T}}^{\mathcal{A}}{\mathcal{M}}=\displaystyle\bigcup_{m\in{\mathcal{M}}}{\mathbf{T}}^{\mathcal{A}}_m{\mathcal{M}}$ and we consider the projection $\hat{{\bf p}}:{\mathbf{T}}^{\mathcal{A}}{\mathcal{M}}\rightarrow {\mathcal{M}}$\\
defined by $\hat{{\bf p}}(m,a,\mu)=m$.\\
We  introduce  the following  context:
\begin{enumerate}
\item
Let $\widetilde{\pi}: \widetilde{\mathcal{A}}\rightarrow {\mathcal{M}}$ be the pull-back of the bundle $\pi:{\mathcal{A}}\rightarrow M$ by ${\bf p}:{\mathcal{M}}\rightarrow M$. We denote by $\widetilde{\bf p}$ the canonical vector bundle such that the following diagram is commutative:
\[
\xymatrix{
           \widetilde{\mathcal{A}}   \ar[r]^{\widetilde{{\bf p}}}\ar[d]_{\widetilde{\pi}} & \mathcal{A} \ar[d]^{\pi}\\
           \mathcal{M} \ar[r]^{{\bf p}}              & M\\
}
\]
\item
Consider the map
\[
\begin{array}
[c]{cccc}
\hat{ \rho}:    & {\mathbf{T}}^{\mathcal{A}}{\mathcal{M}}    & \to 	     & T{\mathcal{M}}		\\
                & (m,a,\mu)  		                         & \mapsto 	 & (m,\mu)
\end{array}
\]
and let $ {\bf p}_{\widetilde{\mathcal{A}}}: {\mathbf{T}}^{\mathcal{A}}{\mathcal{M}}\mapsto \widetilde{\mathcal{A}}$ be the map defined by  ${\bf p}_{\widetilde{\mathcal{A}}}(m,a,v,z)=(m,a)$. Then the following diagrams are commutative
\[
\xymatrix{
           {\mathbf{T}}^{\mathcal{A}}{\mathcal{M}}  \ar[r]^{{{\bf p}}_{\widetilde{\mathcal{A}}}}\ar[d]_{\hat{{\bf p}}} & \widetilde{\mathcal{A}} \ar[d]^{\widetilde{\pi}}\\
           \mathcal{M} \ar[r]^{\operatorname{Id}}             & \mathcal{M}\\
}
\;\;\;\;\;\;\;\;\;\;\;\;
\xymatrix{
           {\mathbf{T}}^{\mathcal{ A}}\mathcal{M}  \ar[r]^{\hat{ \rho}}\ar[d]_{{\bf p}_{\mathcal{A}}} & T\mathcal{M} \ar[d]^{T{\bf p}}\\
           \mathcal{A} \ar[r]^{\rho}              & TM\\
}
\]
\item
If $ {\bf p}_{\mathcal{M}}: {T}{\mathcal{M}}\mapsto {\mathcal{M}}$ is the tangent bundle, consider the associated vertical bundle  ${\bf p}_\mathcal{M}^\mathbf{V} :\mathbf{T}\mathcal{M}\rightarrow \mathcal{M}$.
Then there exists a canonical isomorphism  bundle ${\bf \nu}$ from  the pull-back $\widetilde{\bf p}:\widetilde{\mathcal{E}}\rightarrow \mathcal{M}$ of the the bundle ${\bf p}:\mathcal{E}\rightarrow M$ over ${\bf p}:\mathcal{M}\rightarrow M$ to  ${\bf p}_\mathcal{M}^V :\mathbf{V}\mathcal{M}\rightarrow \mathcal{M}$ so that the following diagram is commutative:
\begin{eqnarray}
\label{eq_nu}
\xymatrix{
           \widetilde{\mathcal{E}}  \ar[r]^{{\bf \nu}} \ar[d]_{\widetilde{\bf p}} & \mathbf{V}\mathcal{M} \ar[d]^{{\bf p}_\mathcal{M}^V }\\
           \mathcal{M} \ar[r]^{\operatorname{Id}}              & \mathcal{M}\\
}
\end{eqnarray}
\end{enumerate}

\begin{theorem}
\label{T_Prolongation}
${}$
\begin{enumerate}
\item
$\hat{\bf p}:{\mathbf{T}}^{\mathcal{A}}{\mathcal{M}}\rightarrow {\mathcal{M}}$ is a convenient bundle with typical fibre $\mathbb{A}\times \mathbb{E}$  and
$(\mathbf{T}^{\mathcal{A}}\mathcal{M},\hat{\bf p},\mathcal{M},\hat{\rho}) $ is an anchored bundle.
\item
${{\bf p}}_{\widetilde{\mathcal{A}}}$ is a surjective bundle morphism whose kernel is a subbundle of $\mathbf{T}\mathcal{M}$. The restriction of $\hat{\rho}$ to $\ker{{\bf p}}_{\widetilde{\mathcal{A}}}$ is a bundle isomorphism on $\mathbf{V}\mathcal{M}$.
\item
Given an open subset $V$, then, for each section $\mathbf{X}$ of $\mathbf{T}\mathcal{M}$ defined on the open set $\mathcal{V}=\hat{\mathbf{p}}^{-1}(V)\subset \mathcal{M}$,   there exists a  pair  $(\mathfrak{a},X)$ of a section $\mathfrak{a}$ of $\widetilde{\mathcal{A}}$  and a vector field $X$ on $\mathcal{V}$ such that
\begin{eqnarray}
\label{eq_aX}       
\forall m\in \mathcal{V},\; T\mathbf{p}(X(m))=\rho \circ {\mathbf{p}_\mathcal{A}}(\mathfrak{a}(m)).
\end{eqnarray}
Conversely such a pair  $(\mathfrak{a},X)$ which  satisfies (\ref{eq_aX}) defines a unique section $\mathbf{X}$  on $\mathcal{V}$, the associated pair of $\mathbf{X}$ is precisely  $(\mathfrak{a},X)$ and  with these notations, we have $\hat{\rho}(\mathbf{X})=X$.
\end{enumerate}
\end{theorem}

\begin{proof} 		
(1) Let $(U,\phi)$ be a chart on $M$ such that we have a trivialization $\tau:\mathcal{A}_U\rightarrow \phi(U)\times \mathbb{A}$ and $\Phi:\mathcal{M}_U\rightarrow \phi(U)\times \mathbb{O}\subset \phi(U)\times \mathbb{E}$. Then $T\phi$ is a trivialization of $TM_U$  on $\phi(U)\times \mathbb{M}$ and $T\Phi$ is a trivialization of $T\mathcal{M}_U$ on $\phi(U)\times \mathbb{O}\times\mathbb{M}\times\mathbb{E}$.\\
To be very precise, according to the notations in  $\S$ \ref{___LocalIdentificationsAndExpressionsInAConvenientBundle}, we have
\begin{description}
\item			
	$\phi(x)\equiv\mathsf{x} $;
\item			
	$\tau(x,a)\equiv (\mathsf{x,a})$;
\item			
	$\Phi(x,e)\equiv (\mathsf{x,e})$ and for $m=(x,e)$, $\Phi(m)\equiv \mathsf{m}$;
\item			
	$T\phi(x,v)\equiv (\mathsf{x,v})$;
\item			
	$T\Phi(m, \mu)= T\Phi(x,e,v,z)\equiv(\mathsf{x,e,v,z})$
\end{description}			
where $\equiv$ stands for  "denoted".\\
By the way, in this local context and with these notations, we have
\begin{eqnarray}
\label{eq_locTAM} 
\mathbf{T}^{\mathcal{A}}\mathcal{M}_U\equiv\{\mathsf{(x,e,a,v, z})\in \phi(U)\times\mathbb{O}\times\mathbb{A}\times \mathbb{M}\times\mathbb{E}\;:\; \mathsf{v={\bf \rho}_x(a)}\}
\end{eqnarray}
where ${\bf \rho}$ corresponds to the local expression of the anchor.
It follows that:
\begin{eqnarray}
\label{eq_LocTAMstrict}
\mathbf{T}^{\mathcal{A}}\mathcal{M}_U\equiv\{\mathsf{(x,e,a,\rho_x(b),z}) \;\;: (\mathsf{x,e,a,z})\in  \phi(U)\times\mathbb{O}\times \mathbb{A}\times \mathbb{E}\},
\end{eqnarray}
and so the map $\mathbf{T}\Phi: \mathbf{T}^{\mathcal{A}}\mathcal{M}_U\rightarrow \phi(U)\times\mathbb{O}\times \mathbb{A}\times \mathbb{E}$  defined by
$\mathbf{T}\Phi(x,e,a,v,z)\equiv(\mathsf{x,e,a,z})$
is a smooth map which is bijective. Moreover, for each $m=(x,e)\in \mathcal{M}_U$ the restriction $\mathbf{T}{\Phi}_m$ to $\mathbf{T}_m\mathcal{M}_U$ is clearly linear and we have the following commutative diagram
\[
\xymatrix{
           \mathbf{T}^\mathcal{A}\mathcal{M}_U  \ar[r]^{\mathbf{T}{\Phi}}\ar[d]_{\hat{\mathbf{p}}} & \phi(U)\times\mathbb{O}\times \mathbb{A}\times\mathbb{E} \ar[d]^{\hat{\pi}_1}\\
           {\mathcal{M}}_U \ar[r]^{\Phi}              & \phi(U )\times\mathbb{O}\\
}
\]
This shows that $\mathbf{T}{\Phi}$ is a local trivialization of $\mathbf{T}^\mathcal{A}\mathcal{M}_U$ modelled on $\mathbb{A}\times \mathbb{E}$.\\
Now, in this local context, $\hat{\rho}(x,e,a,v,z)\equiv (\mathsf{x,e,a,z})$.\\

Consider two such chart domains $U$ and $U'$ in $M$ such that $U\cap U'\not=\emptyset$. Then we have:
\begin{itemize}
\item[--]
the transition maps associated to the trivializations  $\tau$ and $\tau'$ in $\mathcal{A}$ are of type
\[
\mathsf{(x,a)}\mapsto \mathsf{\left( t(x), G_x(a) \right) }
\]
where $\mathsf{x\mapsto G_{x}}$ takes value in $\operatorname{GL}(\mathbb{A})$,  and  is a smooth  map from $U\cap U^\prime$ into  $\operatorname{L}(\mathbb{A})$;
\item[--]
 the  transition maps associated to the trivializations  $\Phi$ and $\Phi'$ in  $\mathcal{M}$
 are of type
\[
\mathsf{(x,e,a,v)} \mapsto \mathsf{\left( t(x),F_x(e) \right) }
\]
where $\mathsf{x} \mapsto \mathsf{F_{x}}$ takes value in $\operatorname{GL}(\mathbb{E})$, and  is a smooth  map from $U\cap U^\prime$ into  $\operatorname{L}(\mathbb{E})$;
\item[--]
if  $\widetilde{\Phi}$ and $\widetilde{\Phi}'$ are the trivializations of $\widetilde{\mathcal{A}}$ associated  to $\Phi$ and $\Phi'$, the  transition maps associated to  trivializations  $\widetilde{\Phi}$ and $\widetilde{\Phi}'$ in $\widetilde{\mathcal{A}}$ are of type
\[
\mathsf{(x,e,a)} \mapsto \mathsf{\left( t(x),F_x(e),G_{x}(a) \right) };
\]
\item[--]
the  transition maps associated to  trivializations  $T\Phi$ and $T\Phi'$ in $T\mathcal{M}$ are of type
\[
\mathsf{(x,e,v,z)} \mapsto \left( \mathsf{ t(x),F_x(e),}  d\mathsf{_xt(y), H_{(x,e)}(v) } \right)
\]
where  $\mathsf{(x,e)} \mapsto \mathsf{H_{(x,e)}}$  takes value in $\operatorname{GL}(\mathbb{E})$ and is a smooth map from $U\cap U^\prime$ to $\operatorname{L}(\mathbb{E})$;
\item[--]
the transition maps associated to  trivializations  ${\bf T}\Phi$ and ${\bf T}\Phi'$ in ${\bf T}^\mathcal{A}\mathcal{M}$ are of type
\[
\mathsf{(x,e,a,z)} \mapsto \mathsf{\left( t(x),F_x(e),G_{x}(a), H_{(x,e)}(z)) \right) }.
\]
\end{itemize}
Clearly, this implies that $\hat{\bf p}:{\bf T}\mathcal{M}\to \mathcal{M}$ is a convenient bundle.\\

Now, in a  trivialization $\tau$  and $T\phi$,  we write  $\rho(x,a)\equiv (\mathsf{x,a}) \mapsto (\mathsf{x,\rho_x(a)})$.
If, in another trivialization $\tau'$ and $ T\phi'$, we write $\rho(x,a)\equiv (\mathsf{x',a'}) \mapsto (\mathsf{x',\rho'_{x'}(a')})$, then, for the  associated transition maps, we have
\[
\mathsf{(x',\rho'_{x'})= (x',} d \mathsf{_x t\circ \rho_x\circ G_x.)}
\]
It follows easily that $\hat{\rho}$ is a bundle convenient morphism.\\

(2) By construction, the following diagram is commutative:
\[
\xymatrix{
           {\bf T}^{\mathcal{A}}\mathcal{M} \ar[r]^{\hat{\rho}}\ar[d]_{{\bf p}_{\widetilde{\mathcal{A}}}} & T{\mathcal{M}} \ar[d]^{p_\mathcal{M}}\\
           \widetilde{\mathcal{A}} \ar[r]^{\widetilde{\pi}}              & \mathcal{M}\\
}
\]		

In the trivialization  $\widetilde{\Phi}:\widetilde{A}_{\mathcal{M}_U}\rightarrow \phi(U)\times\mathbb{O}\times \mathbb{A}$, using the same convention as previously, we have
\[
\{(x,e,a,v,z)\mapsto {{\bf p}}_{\widetilde{\mathcal{A}}}(x,e,a,v)\}\equiv \{(\mathsf{x,e,a,v,z})\mapsto (\mathsf{x,e,a})\}.
\] 			
Thus, by analog arguments as in the proof of (1),  it is clear that	${{\bf p}}_{\widetilde{\mathcal{A}}}$
is compatible with the transition maps associated to the trivializations over the chart domains $U$ and $U'$ of $M$ for ${\bf T}^{\mathcal{A}}\mathcal{M}$ and for  $\widetilde{\mathcal{A}}$.	
Thus ${\bf p}_{\widetilde{\mathcal{A}}}:  {\bf T}^{\mathcal{A}}\mathcal{M}\to \widetilde{\mathcal{A}}$
 is a Banach bundle morphism which is surjective.\\
\noindent From the construction of $\mathbf{T}^{\mathcal{A}}\mathcal{M}$ we have
\[
\ker{{\bf p}}_{\widetilde{\mathcal{A}}}=\{(m,a,v)\in \mathbf{T}_m^\mathcal{A}\mathcal{M}\;\;: \rho(a)=0\}.
\]
The definition of $\hat{\rho}$ implies  that its restriction to $\ker{{\bf p}}_{\widetilde{\mathcal{A}}}$ is an isomorphism onto $\mathbf{V}\mathcal{M}$, which ends the proof of (2).\\

(3) Let $\mathbf{X}$  be a section of $\mathbf{T}\mathcal{M}$ defined on $\mathcal{V}=\hat{\mathbf{p}}^{-1}(V)$. According to (2), if  we set $\mathfrak{a}=\mathbf{p}_{\widetilde{\mathcal{A}}}\circ \mathbf{X}$ and $X=\hat{\rho}(\mathbf{X})$, the pair  $(\mathfrak{a},X)$ is well defined and  from the definition of $\mathbf{T}^\mathcal{A}\mathcal{M}$ the relation (\ref{eq_aX}) is satisfied. Conversely, if $\mathfrak{a}$ is a section $\mathfrak{a}$ of $\widetilde{\mathcal{A}}$  and $X$ a vector field  on $\mathcal{V}$, the relation (\ref{eq_aX}) means exactly that $\mathbf{X}(m)=(\widetilde{\mathbf{p}}(\mathfrak{a}(m)), X(m))$  belongs to $\mathbf{T}_m^\mathcal{A}\mathcal{M}$; so we get a section $\mathbf{X}$ of $\mathbf{T}^\mathcal{A}\mathcal{M}$. Now it is clear that    $\mathfrak{a}=\mathbf{p}_{\widetilde{\mathcal{A}}}\circ \mathbf{X}$ and $X=\hat{\rho}(\mathbf{X})$.	
\end{proof}
	
\begin{definition}
\label{D_ProlongantionOfAnAnchorBundleOverAFibration}
The anchored bundle $(\mathbf{T}^{\mathcal{A}}\mathcal{M}, \hat{\mathbf{p}}, \mathcal{M}, \hat{\rho})$ is called the prolongation of $(\mathcal{A},\pi,M,\rho)$ over $\mathcal{M}$. The subbundle $\ker{{\bf p}}_{\widetilde{\mathcal{A}}}$ will be denoted $\mathbf{V}^{\mathcal{A}}\mathcal{M}$ and is called the vertical subbundle.
\end{definition}
	
\begin{remark}
\label{R_VerticalBundle}
According to the proof of Theorem \ref{T_Prolongation}, if ${\bf V}^{\mathcal{A}}\mathcal{M}_U$ is the restriction of ${\bf V}^{\mathcal{A}}\mathcal{M}$ to $\mathcal{M}_U$, we have
\[
{\bf V}^{\mathcal{A}}\mathcal{M}_U\equiv\{\mathsf{(x,e,a,0,z}) \;\;: (\mathsf{x,e,a,z})\in  \phi(U)\times\mathbb{O}\times \mathbb{A}\times \mathbb{E}\}.
\]
\end{remark}

\begin{examples}
\label{Ex_Class}
For simplicity in these examples we assume that the model space $\mathbb{M}$ of $M$ and the typical fiber $\mathbb{A}$ of $\mathcal{A}$ are  Banach spaces.
\begin{enumerate}
\item
If $\mathcal{A}={\mathcal{M}}=TM$ then we have  ${\mathbf{T}}{\mathcal{A}}=TTM$ and  ${\mathbf{T}}{\mathcal{A}}^*=TT^{*}M$ and the anchor $\hat{\rho}$ is the identity.
\item
If $\mathcal{A}={\mathcal{M}}$ then $\mathbf{T}^{\mathcal{A}}\mathcal{A}$ is simply denoted $\mathbf{T}\mathcal{A}$ 
				and the anchor  $\hat{\rho}$ is the map $(x,a,b,c)\mapsto (x,a,\rho(b),\nu(c))$ from  $\mathbf{T}\mathcal{A}$ to $T\mathcal{A}$.
\item
If $\mathcal{M}=\mathcal{A}^*$  then $\mathbf{T}^{\mathcal{A}}\mathcal{A}^*$ is simply denoted $\mathbf{T}\mathcal{A}^*$ and the anchor  $\hat{\rho}$ is the map $(x,\xi,a,\eta)\mapsto (x,\xi,\rho(a),\nu(\eta)$ from $\mathbf{T}\mathcal{A}^*$ to $T\mathcal{A}^*$.			
\item
If $\mathcal{M}=\mathcal{A}\times_M {\mathcal{A}}^*$ then  $\mathbf{T}^{\mathcal{A}}\mathcal{M}$ is simply denoted $\mathbf{T}(\mathcal{A}\times\mathcal{A}^*)$ and the anchor  $\hat{\rho}$ is the map $(x,a,\xi,b,c,\eta)\mapsto (x,a,\xi,\rho(b),\nu(v),\eta))$ from $\mathbf{T}(\mathcal{A}\times\mathcal{A}^*)$ to  $T(\mathcal{A}\times\mathcal{A}^*)$.		 
\item
If $\mathcal{M}$ is a conic submanifold of $\mathcal{A}$ (cf \cite{Pe2}), then $\mathbf{T}^{\mathcal{A}}\mathcal{M}$ is simply denoted $\mathbf{T}\mathcal{M}$ and the anchor $\hat{\rho}$ is the same expression as in (2).
\end{enumerate}
\end{examples}

\begin{remark}
\label{R_Chart} 
We come back to the previous general context: $(\mathcal{A},\pi, M,[.,.]_\mathcal{A})$ is a convenient Lie algebroid and $\mathbf{p}:\mathcal{E}\to M$is  a convenient vector bundle and $\mathcal{M}$ is an open submanifold of $\mathcal{E}$ which is fibered on $M$.\\
Let  $(U, \phi)$ be a chart of $M$ such that $\mathcal{A}_U$  (resp. $\mathcal{E}_U$) is trivial and we have denoted
\begin{center}
$\tau:\mathcal{A}_U\rightarrow \phi(U)\times \mathbb{A}$ (resp. $\Phi : \mathcal{E}_U\rightarrow \phi(U)\times \mathbb{E}$)
\end{center}
 the associated trivialization (cf. notations in the proof of Theorem  \ref{T_Prolongation}). Then we have a canonical anchored bundle			
\begin{center}
$(\phi(U)\times \mathbb{A}, \pi_1,\phi(U), \mathsf{r}= T\phi\circ \rho \circ\phi^{-1})$
\end{center}
where $\pi_1:\phi(U)\times \mathbb{A}\rightarrow \phi(U)$.\\			
Therefore,  the prolongation  $\phi(U)\times \mathbb{A}$ over $\phi(U)\times \mathbb{E}$ is then
\[
\mathbf{T}^{\phi(U)\times \mathbb{A}}(\phi(U)\times \mathbb{E})=\phi(U)\times \mathbb{E}\times\mathbb{A}\times\mathbb{E}
\]
and the anchor $\hat{\mathsf{r}}$  is the map $(\mathsf{x,e,a,z}) \mapsto  (\mathsf{x,e, r(a), z})$.\\
Note that the bundle $\mathbf{T}^{\phi(U)\times \mathbb{A}}(\phi(U)\times \mathbb{E})=\phi(U)\times \mathbb{E}\times\mathbb{A}\times\mathbb{E}$ can be identified with the subbundle
\[
\left\{ (\mathsf{x,e,a,0,v})\in U\times \mathbb{E}\times\mathbb{A}\times\mathbb{M}\times\mathbb{E} \right\}
\]
of $T(\phi(U)\times \mathbb{E})=\phi(U)\times \mathbb{E}\times\mathbb{A}\times\mathbb{M}\times\mathbb{E}.$\\
An analog description  is true for any open set $\mathcal{U}$ in $\mathcal{E}$ such that $\mathbf{p}(\mathcal{U})=U$ since $\mathcal{U}$ is contained in $\mathcal{M}_U$.\\
\end{remark}
		
Given a fibred  morphism  $\Psi:{\mathcal{M}}\rightarrow {\mathcal{M}}'$ between the fibred manifolds ${\bf p}:{\mathcal{M}}\rightarrow M$  and ${\bf p}':{\mathcal{M}'}\rightarrow M'$  over $\psi:M\rightarrow M'$ and a morphism of anchored bundles $ \varphi$ between  $({\mathcal{A}},\pi, M,\rho)$ and $({\mathcal{A}}',\pi', M',\rho')$,  over  $\psi:M\rightarrow M'$, we get a map
\begin{eqnarray}\label{eq_bfTPsi}
\begin{array}[c]{cccc}
\label{eq_TPsi}
{\mathbf{T}}\Psi:  & {\mathbf{T}}^{\mathcal{A}}{\mathcal{M}}& \to & {\mathbf{T}}^{\mathcal{A}'}{\mathcal{M}}'\\
    & (m,a,\mu)  	& \mapsto 	& ( \Psi(m);\varphi(a),T_m\Psi(\mu) )
\end{array}
\end{eqnarray}

\begin{remark}
\label{R_chartTM}
As in the proof of Theorem  \ref{T_Prolongation},  consider a chart  $(\mathcal{U},\Phi)$  of $\mathcal{M}$ where $\mathbf{p}(\mathcal{U})=U$ and let $\tau:\mathcal{A}_U\rightarrow U\times\mathbb{A}$ be an associated trivialization of $\mathcal{A}_U$.  According to Remark \ref{R_Chart},   since $\Phi$ is a smooth diffeomorphism from $\mathcal{U}$ onto its range $\phi(U)\times\mathbb{O}$ in $\phi(U)\times\mathbb{E}$, then $\mathbf{T}\Phi: \mathbf{T}^{\mathcal{A}}\mathcal{M}_{| \mathcal{U}}\rightarrow \mathbf{T}^{\phi(U)\times \mathbb{A}}\mathcal{U}=\Phi(\mathcal{U})\times\mathbb{A}\times\mathbb{E}$ is a bundle isomorphism  and so  $(\mathbf{T}^{\mathcal{A}}\mathcal{M}_{| \mathcal{U}},\mathbf{T}\Phi)$ is a chart for $\mathbf{T}^\mathcal{A}\mathcal{M}$.
\end{remark}
		
\begin{notations}
\label{N_Prolongations}
From now on, the anchored bundle $(\mathcal{A},\pi,M,\rho)$ is fixed  and, if no confusion is possible, we simply denote  by $\mathbf{T}\mathcal{M}$ and $\mathbf{V}\mathcal{M}$   the sets $\mathbf{T}^{\mathcal{A}}\mathcal{M}$    and $\mathbf{V}^\mathcal{A}\mathcal{M}$  respectively. In particular {\bf when $\mathcal{M}=\mathcal{A}$, the prolongation $\mathbf{T}\mathcal{A}$ will be simply called the prolongation of the  the Lie algebroid $\mathcal{A}$}\index{prolongation of a convenient Lie algebroid}. The bundle $\mathbf{V}\mathcal{M}$  will be considered as a subbundle of $\mathbf{T}\mathcal{M}$ as well as of $T\mathcal{M}$.
\end{notations}

\subsection{Connections on a prolongation}
\label{___ConnectionsOnAProlongation}
Classically (\cite{KrMi}, 37), a \emph{connection}\index{connection}
\footnote{In the finite dimensional context such a connection is sometimes called a nonlinear connection.} on a convenient vector bundle $\mathbf{p}:\mathcal{E}\rightarrow M$ is a Whitney decomposition $T\mathcal{E}=H\mathcal{E}\oplus V\mathcal{E}$.
Now, as in finite dimension  we introduce this notion on $\mathbf{T}\mathcal{M}$.
\begin{definition}
\label{D_NLConnection}
A connection on $\mathbf{T}\mathcal{M}$ is  a decomposition of this bundle in a Whitney sum $\mathbf{T}\mathcal{M}=\mathbf{H}\mathcal{M}\oplus \mathbf{V}\mathcal{M}$.\end{definition}
Such a decomposition is equivalent to  the datum of an endomorphism $\bf{N}$ of  $\mathbf{T}\mathcal{M}$ such that $\mathbf{N}^2=\operatorname{Id}$   with $\mathbf{T}\mathcal{M}=\operatorname{ker} (\operatorname{Id}+\mathbf{N})$ and $\mathbf{H}\mathcal{M}=\operatorname{ker} (\operatorname{Id}-\mathbf{N})$ where $\operatorname{Id}$ is the identity morphism of $\mathbf{T}\mathcal{M}$. We naturally get two projections:
\begin{description}
\item[ ]
$h_\mathbf{N} =\displaystyle\frac{1}{2}(\operatorname{Id}+\mathbf{N}): \mathbf{T}\mathcal{M}\rightarrow \mathbf{H}\mathcal{M}$
\item[ ]
$v_\mathbf{N}=\displaystyle\frac{1}{2}(\operatorname{Id}-\mathbf{N}): \mathbf{T}\mathcal{M}\rightarrow \mathbf{V}\mathcal{M}$.
\end{description}
$v_\mathbf{N}$ and  $h_\mathbf{N}$ are called respectively  the \emph{vertical}\index{projector!vertical} and \emph{horizontal} projector\index{projector!horizontal} of $\mathbf{N}$.\\
Using again the context of Remark \ref{R_chartTM}, we have  charts
		$$\Phi:\mathcal{M}_U\rightarrow \phi(U)\times \mathbb{O}\;\;\textrm{ and}\;\;\mathbf{T}\Phi:\mathbf{T}\mathcal{M}_U\rightarrow \phi(U)\times \mathbb{O}\times \mathbb{A}\times\mathbb{E}.$$
If $\mathbf{N}$ is a  connection on $\mathbf{T}\mathcal{M}$, then  $\mathsf{N}=\mathbf{T}\Psi\circ \mathbf{N}\circ \Psi^{-1}$ is a non linear connection on the trivial bundle $ \phi(U)\times \mathbb{O}\times \mathbb{A}\times\mathbb{E}$. Thus $\mathsf{N}$ can be written as a matrix field of endomorphisms of $\mathbb{A}\times\mathbb{E}$ of type		
\begin{eqnarray}
\label{eq_Christoffel}
\begin{pmatrix}
\operatorname{Id}_\mathbb{A}	&	0\\
-2\digamma				& 	-\operatorname{Id}_\mathbb{E}
\end{pmatrix}
\end{eqnarray}
and so the associated horizontal (resp. vertical) projector is given by
\begin{description}
\item
$\mathsf{h_N}_{\mathsf{m}}(\mathsf{a,z})=\displaystyle\frac{1}{2}(\mathsf{a,z})- \digamma(\mathsf{a})$;
\item
$\mathsf{v_N}_{\mathsf{m}}(\mathsf{a,z})=\displaystyle\frac{1}{2}(\mathsf{a,v})+\digamma(\mathsf{a})$.
\end{description}
The associated horizontal space in $\{\mathsf{m}\}\times \mathbb{A}\times \mathbb{E}$ is
\[
\left\{ \displaystyle\frac{1}{2}(\mathsf{a,z})- \digamma(\mathsf{a}), (\mathsf{a,z})\in \{\mathsf{m}\}\times \mathbb{A}\times\mathbb{E}\right\}
\]
and the  associated vertical space in $\{\mathsf{m}\}\times \mathbb{A}\times \mathbb{E}$ is $\{\mathsf{m}\}\times \{\mathsf{0}\}\times \mathbb{E}$.\\
$\digamma$ is called the \emph{(local) Christoffel symbol of} ${\bf N}$.\\
Let $\widetilde{\bf p}: \widetilde{\mathcal{A}\times\mathcal{E}}\rightarrow \mathcal{M}$ the fibered product bundle over $\mathcal{M}$  of
$({\pi},{\bf p}): {\mathcal{A}}\times {\mathcal{E}}\rightarrow {M}$.  We have natural  inclusions $\iota_1 :\widetilde{\mathcal{A}}\rightarrow \widetilde{\mathcal{A}\times\mathcal{E}}$ and
		$\iota_2 :\widetilde{\mathcal{E}}\rightarrow \widetilde{\mathcal{A}\times\mathcal{E}}$, given respectively by $\iota_1(m,a)=(m,a,0)$ and $\iota_1(m,z)=(m,0,z)$,  such that
\begin{eqnarray}
\label{eq_DecompAE}
\widetilde{\mathcal{A}\times\mathcal{E}}=\iota_1 (\widetilde{\mathcal{A}})\oplus \iota_2 (\widetilde{\mathcal{E}}).
\end{eqnarray}
With these notations, we have
\begin{proposition}
\label{P_IsoConnection}
${}$
\begin{enumerate}
\item
There exists a non linear connection   $\mathbf{N}$ on $\mathbf{T}\mathcal{M}$ if and only if there exists a convenient bundle morphism $\mathbf{H}$ from   $\widetilde{\mathcal{A}}$ to $\mathbf{T}\mathcal{M}$ such that  $\mathbf{T}\mathcal{M}=\mathbf{H}(\widetilde{\mathcal{A}})\oplus \mathbf{V}\mathcal{M}$. In this case $\mathbf{T}\mathcal{M}$ is isomorphic to $\widetilde{\mathcal{A}\times\mathcal{E}}$.
\item
Assume that $\mathbf{N}$ is a connection on ${\bf T}\mathcal{M}$. Let  $\Upsilon$ be semi-basic vector valued   $\Upsilon$ \footnote{ that is a morphism from ${\bf T}\mathcal{M}$ to ${\bf V}\mathcal{M}$ such that $\Upsilon(\mathbf{Z})=0$ for any local vertical section $\mathbf{Z}$}, then $\mathbf{N}+\Upsilon$ is a connection on ${\bf T}\mathcal{M}$. Conversely, given any nonlinear connection $\mathbf{N}'$ on ${\bf T}\mathcal{M}$, there exists a unique  semi-basic  vector valued   $\Upsilon$  such that $\mathbf{N}'=\mathbf{N}+\Upsilon$.
\end{enumerate}
\end{proposition}
According to this Proposition we introduce:
\begin{definition}
\label{D_SplitProlongation}
A prolongation of $\mathcal{A}$ over ${\bf p}:\mathcal{M}\to M$ is called a split prolongation\index{split prolongation} if there exists a Withney decomposition ${\bf T}\mathcal{M}={\bf K}\mathcal{M}\oplus {\bf V}\mathcal{M}$.
\end{definition}

The following result is a clear consequence of Proposition \ref{P_IsoConnection}.

\begin{corollary}
\label{C_SplitProlongation}
Let ${\bf T}\mathcal{M}$ be a split prolongation. Then there exists a connection on  ${\bf T}\mathcal{M}$  and ${\bf T}\mathcal{M}$ is isomorphic to $\widetilde{\mathcal{A}\times\mathcal{E}}$.
\end{corollary}

\begin{proof}[Proof of Proposition \ref{P_IsoConnection}]${}$\\
(1) Assume that we have a  connection $\mathbf{N}$ on $\mathbf{T}\mathcal{M}$ and let  $\mathbf{T}\mathcal{M}=\mathbf{H}\mathcal{M}\oplus \mathbf{V}\mathcal{M}$ be the associated Whitney decomposition.\\
Let $\mathbf{H}$ be the restriction of $\mathbf{p}_{\widetilde{\mathcal{A}}}$ to $\mathbf{H}\mathcal{M}$. Since $\mathbf{V}\mathcal{M}$ is the kernel of  the surjective morphism $\mathbf{p}_{\widetilde{\mathcal{A}}}$, it follows that $\mathbf{H}$ is an isomorphism.  Since we have an isomorphism $\nu:\widetilde{\mathcal{E}}\rightarrow \mathbf{V}\mathcal{M}$, according to  (\ref{eq_DecompAE}), it follows easily that $\widetilde{\mathcal{A}\times\mathcal{E}}$ is isomorphic to $\mathbf{T}\mathcal{M}$. The converse is clear.\\
			
(2) At first, if $\Upsilon$ is semi basic, then $\operatorname{ker} (\operatorname{Id}+\mathbf{N}+\Upsilon)=\mathbf{V}\mathcal{M}$ and clearly the range of $\operatorname{Id}+\mathbf{N}+\Upsilon$ is a supplemented subbundle of $\mathbf{V}\mathcal{M}$.\\
On the one hand, if $\mathbf{N}'$ is a connection, we set $\Upsilon=\mathbf{N}'-\mathbf{N}$. Then $\Upsilon(\mathbf{Z})=0$, for all local vertical sections. On the other hand  $\Upsilon(\mathbf{X})=(\operatorname{Id} +\mathbf{N}')(\mathbf{X})-(\operatorname{Id} +\mathbf{N}(\mathbf{X}) $ which belongs to $\mathbf{V}\mathcal{M}$.\\
\end{proof}	

A sufficient condition for the existence of a connection on $\mathbf{T}\mathcal{M}$ is given by the following result:
\begin{theorem}
\label{T_CShconnection}
Assume that there exists a linear connection on the bundle $\mathbf{p}:\mathcal{E}\rightarrow M$. Then there exists a  connection $\mathbf{N}$ on $\mathbf{T}\mathcal{M}$.
\end{theorem}
		
\begin{proof}
 let  $\mathbf{p}^{\ast}TM$ (resp. $\mathbf{p}^{\ast}\mathcal{E}$) is the pull-back over $\mathcal{M}$ of $TM\rightarrow M$ (resp. $\mathcal{E}\rightarrow M$).\\
 If  there exists a linear connection $N$ on the bundle $\mathcal{E}\rightarrow M$,  there exists a convenient  isomorphism bundle
\[
\kappa=(\kappa_1,\kappa_2): T\mathcal{E}\rightarrow \mathbf{p}^{\ast}TM\oplus  \mathbf{p}^{\ast}\mathcal{E}.\]
(cf.  Theorem 3.1 \cite{AgSu} in Banach setting,  \cite{BCP}, Chapter 6 in convenient setting).
Therefore, for an open  fibred submanifold $\mathcal{M}$ of $\mathcal{E}$, by restriction,  we obtain an isomorphism (again denoted $\kappa$).
\[
\kappa:T\mathcal{M}\rightarrow \mathbf{p}^{\ast}TM \oplus  \mathbf{p}^{\ast}\mathcal{E}
\]

Without loss of generality, we can identify $T\mathcal{M}$ with $\mathbf{p}^{\ast}TM \oplus  \mathbf{p}^{\ast}\mathcal{E}$. For $m=(x,e)$, the fibre $\mathbf{T}_m\mathcal{M}$ is then
\[
\mathbf{T}_m\mathcal{M}=\{(a,\mu)\in \mathcal{A}_x\times T_xM\times\mathbf{V}_m\mathcal{M}\;\;: \rho(a)=T\mathbf{p}(\mu)\}.
\]
But, under our identification of $T\mathcal{M}$ with $\mathbf{p}^{\ast}TM\oplus  \mathbf{p}^{\ast}\mathcal{E}$, if $m=(x,e)$, this implies that $\mu\in T_m\mathcal{M}$, can be written as a pair   $(v,z)\in T_xM\times \mathcal{E}_x$ and so we replace the condition  $\rho(a)=T\mathbf{p}(\mu)$ by $\rho_x(a)=v$.\\
Recall that, in the one hand, we have a chart (cf. proof of Theorem \ref{T_Prolongation}):
\[
\mathbf{T}\Phi: \mathbf{T}\mathcal{M}_U\rightarrow \Phi(\mathcal{M}_U)\times \mathbb{A}\times \mathbb{E}.
\]
The value of $\mathbf{T}\Phi(m, a,\mu)$ can be written $(\phi(x),\Phi_x(e),\tau_x(a),T_m\Phi(\mu))$  (value denoted $(\mathsf{x,e,a,z})$)  with $T_m\Psi (T\mathbf{p}(\mu))=T_x\phi(\rho_x(a))$. But, under our assumption,
we have $T_m\Phi(\mu))=(T_x\phi(\rho_x(a)),\mathsf{z})\in \{(\mathsf{x,e})\}\times \mathbb{M}\times \mathbb{E}$ and so we obtain
\begin{eqnarray}
\label{eq_TPhiassump}
\mathbf{T}\Phi(m, a,\mu)\equiv (\mathsf{x,e,a,z}).
\end{eqnarray}
On the other hand, we have a trivialization  $\widetilde{\tau\times\Phi}$ from $\widetilde{\mathcal{A}\times\mathcal{E}}_{\mathcal{M}_U}$ to $\Phi(\mathcal{M}_U)\times\mathbb{A}\times\mathbb{E}$ over $\Phi$.  In fact, we have $(\widetilde{\tau\times\Phi})(x,e,a,z)=(\phi(x),\Phi_x(e),\tau_x(a),\Phi_x(z))$. \\
According to our assumption,  the map $\widetilde{\Psi}:\widetilde{\mathcal{A}\times}\mathcal{E}\rightarrow \mathbf{T}\mathcal{M}$ is given by
\[
\widetilde{\Psi}(m,a,z)=(m,a,\rho(a),z)
\]
is well defined. In local coordinates, we have
\[
\widetilde{\Psi}(e,x,a,\rho(a),z)\equiv(\mathsf{x,e,a,z}).
\]
Thus $\widetilde{\Psi}$ is the identity in local coordinates and so is a local bundle isomorphism. To complete the proof, we only have to show that under our assumption $\widetilde{\Psi}$ is a convenient bundle morphism.\\
By analogy with the notations used in the proof of Theorem \ref{T_Prolongation} (1),  let  $(U',\phi')$ be  another chart on $M$ and consider all the corresponding  trivializations $T\phi'$, $\Phi'$, and  $\Phi'$. We set  $(\mathsf{x',e',v',z'})=T\Phi'(\mathfrak{e},\mathfrak{z})$ (i.e. $(\mathfrak{e},\mathfrak{z})\equiv(\mathsf{x',e',v',z'}) $ following our convention), in these new coordinates, we have $\widetilde{\Psi}(\mathfrak{e},\mathfrak{z})  \equiv(\mathsf{x',e',a',z'})$. Assume that $U\cap U'\not=\emptyset$.  For the change of coordinates,  we set $\theta_\mathsf{x}=\phi'\circ\phi^{-1}(\mathsf{x})$, and each associated  transition map gives rise to a smooth field of isomorphisms of convenient spaces as follows:
\begin{eqnarray*}
\begin{aligned}
&T_\mathsf{x}\theta(\mathsf{v})=T_\mathsf{x}(\phi'\circ\phi^{-1})(\mathsf{v})\\
&\mathfrak{T}_\mathsf{x}(\mathsf{a})=\left(\tau' \circ\tau ^{-1}\right)_\mathsf{x}(\mathsf{a})\\
&\Theta_\mathsf{x}(\mathsf{e})=(\Phi'\circ \Phi^{-1})_\mathsf{x}(\mathsf{e}).
\end{aligned}
\end{eqnarray*}

Thus, under our assumption, according to (\ref{eq_TPhiassump}), in fact we have
\[
\mathbf{T}\Phi'\circ \mathbf{T}\Phi^{-1}(\mathsf{x,e,a,z})=(\theta(\mathsf{x}),\Theta_\mathsf{x}(\mathsf{e}), \mathfrak{T}_\mathsf{x}(\mathsf{a}), \Theta_\mathsf{x}(\mathsf{z})).
\]
Now with the previous notations we have
\[
\left(\widetilde{\tau'\times\Phi'}\right)\circ \left(\widetilde{\tau\times\Phi}\right)^{-1}(\mathsf{x,e,a,z})=(\theta(\mathsf{x}),\Theta_\mathsf{x}(\mathsf{e}), \mathfrak{T}_\mathsf{x}(\mathsf{a}), \Theta_\mathsf{x}(\mathsf{z}))=\mathbf{T}\Phi'\circ \mathbf{T}\Phi^{-1}(\mathsf{x,e,a,z}).
\]
Since, in such local coordinates, $\widetilde{\Psi}$ is the identity map, so under our assumption $\widetilde{\Psi} $ is a convenient bundle isomorphism.
\end{proof}

\subsection{Prolongation of the Lie bracket}
\label{___ProlongationOfTheLieBracket}
In the finite dimensional framework,   a Lie bracket for smooth sections of   $\hat{\bf p}:{\bf T}\mathcal{M}\to \mathcal{M}$ is well defined. Unfortunately, we will see that it is  not true in this general convenient context.\\

According to the notations \ref{N_locM}, for each open set $U$ of $M$, we denote by $\Gamma(\textbf{T}\mathcal{M}_U)$,   $\Gamma(\textbf{V}\mathcal{M}_U)$ and $\Gamma(\widetilde{\mathcal{A}}_U)$ the $C^\infty(\mathcal{M}_{U})$-module of sections of $\textbf{T}\mathcal{M}_U$, ${\bf V}\mathcal{M}_U$ and $\widetilde{\mathcal{A}}_U$ respectively. We also denote $\Gamma(\mathcal{A}_U)$   the $C^\infty({U})$-module of sections  $\pi:\mathcal{A}_U\to U$.
\begin{definition}
\label{D_Mproj}
Let $U$ be an open subset of $M$.
A section  $\mathbf{X}$ in $\Gamma({\bf T}\mathcal{M}_U)$ is called projectable\index{projectable section}  if  there exists $\mathfrak{a}\in\Gamma({\mathcal{ A}}_U)$  such that
\[
{\bf p}_{\mathcal{A}}\circ\mathbf{ X}=\mathfrak{a}.
\]
\end{definition}

Therefore  $\mathbf{ X} $  is projectable if and only if there exists a vector field  $X$ on $\mathcal{M}_U$ and $\mathfrak{a}\in\Gamma(\mathcal{ A}_U)$ such that  (see  Theorem \ref{T_Prolongation}) 
\[
\mathbf{ X}=(\mathfrak{a}\circ{\bf p},X) \textrm{ with   } T{\bf p}(X)=\rho\circ \mathfrak{a}.
\]

\textbf{Assume now that  $(\mathcal {A},\pi,M,\rho,[.,.]_{\mathcal {A}})$ is a convenient Lie algebroid}.
\smallskip

Let  $\mathbf{X}_i=(\mathfrak{a}_i\circ{\bf p},X_i)$, $i \in \{ 1,2 \}$ be two projectable sections defined on $\mathcal{M}_U$. We set
\begin{eqnarray}
\label{eq_projectTMbracket}
[\mathbf{X}_1,\mathbf{X}_2]_{{\bf T}\mathcal{M}}=([\mathfrak{a}_1,\mathfrak{a}_2]_{\mathcal{A}}\circ {\bf p},[X_1,X_2])
\end{eqnarray}
Since $\rho([\mathfrak{a}_1,\mathfrak{a}_2]_{\mathcal{A}})=[\rho(\mathfrak{a}_1),\rho(\mathfrak{a}_2)]$ and $T{\bf p} \left( [X_1,X_2] \right) =[T{\bf p}(X_1),T{\bf p}(X_2)]$, it follows that  $[\mathbf{X}_1,\mathbf{X}_2]_{{\bf T}\mathcal{M}}$ is a projectable section which is well defined. Moreover, we have
\begin{eqnarray}\label{eq_HatRhoMorphism}
\hat{\rho} \left( [\mathbf{X}_1,\mathbf{X}_2]_{{\bf T}\mathcal{M}} \right) =[\hat{\rho}(\mathbf{X}_1),\hat{\rho}(\mathbf{X}_2)].
\end{eqnarray}

\begin{comments}
\label{Com_PartialBracket}
Now, \textbf{in finite dimension}, it is well known that  the module of sections of $\widetilde{\mathcal{A}}_U\to \mathcal{M}_U$ is  the $C^\infty(\mathcal{M}_{U})$-module generated by the set of sections $\mathfrak{a}\circ {\bf p}$ where $\mathfrak{a}$ is any section of $\mathcal{A}_U\to U$.
 Therefore, according to Theorem \ref{T_Prolongation}, the module $\Gamma({\bf T}\mathcal{M}_U)$ is generated, as $C^\infty(\mathcal{M}_{U})$-module, by the set of all projectable sections of $\Gamma({\bf T}\mathcal{M}_U)$. This result is essentially a consequence of the fact  that, in the local context used in the proof of Theorem \ref{T_Prolongation}, over such a chart domain $U$, the bundle $\widetilde{\mathcal{A}}_U$ is a finite dimensional bundle and so the module of sections of $\widetilde{\mathcal{A}}_U$ over $\mathcal{M}_U$ is finitely generated as $C^\infty(\mathcal{M}_U)$. Thus, if $\mathcal{A}$ is a finite rank bundle over $M$, then the module $\Gamma({\bf T}\mathcal{M}_U)$ is generated, as $C^\infty(\mathcal{M}_{U})$-module, by the set of all projectable sections of $\Gamma({\bf T}\mathcal{M}_U)$.\\
 \textbf{Unfortunately, this is no longer  true in the convenient context and even in the Banach setting in general}.\\
Note that, under some type of approximation properties of $\mathbb{A}$, we can show that the module generated by the set of all projectable sections of $\Gamma(\mathcal{A}_U)$  as $C^\infty(\mathcal{M}_{U})$-module, is dense in  $\Gamma(\widetilde{\mathcal{A}}_U)$  as a convenient space. In this case,  the $C^\infty(\mathcal{M}_{U})$-module, generated by the set of all projectable sections of $\Gamma({\bf T}\mathcal{M}_U)$ will be dense in $\Gamma({\bf T}\mathcal{M}_U)$ (as a convenient space). We could hope that, in this context, the Lie bracket $[.,.]_{\mathbf{T}\mathcal{M}_U}$ can be extended to  $\Gamma({\bf T}\mathcal{M}_U)$.\\
\end{comments}

\begin{definition}
We denote by $\mathfrak{P}({\bf T}\mathcal{M}_U)$ the $C^\infty(\mathcal{M}_{U})$-module $\Gamma({\bf T}\mathcal{M}_U)$   generated by the set of projectable sections defined on $U$ in the $C^\infty(\mathcal{M}_{U})$-module.
\end{definition}
Each module  $\mathfrak{P}({\bf T}\mathcal{M}_U)$ has the following properties:
\begin{lemma}
\label{L_extbracket}
${}$
\begin{enumerate}
\item
For any open subset $U$ in $M$,  there exists a well defined Lie bracket $[.,.]_{{\bf T}\mathcal{M}_U}$ on $\mathfrak{P}({\bf T}\mathcal{M}_U)$ which  satisfies the assumption of Definition \ref{D_AlmostLieBracketOnAnAnchoredBundle}  and whose restriction to projectable sections is given by the relation (\ref{eq_projectTMbracket})
\item
For each $x\in M$, there exists a chart domain $U$ around $x$ in $M$ such that ${\bf T}\mathcal{M}_U$ is trivializable over $\mathcal{M}_U$ and for each $(m,v)\in {\bf T}_m\mathcal{M}$ where $m=(x,e)\in \mathcal{M}_U$, there exists a projectable section ${\bf X}$ defined on $\mathcal{M}_U$ such that ${\bf X}(m)=(a,v)$.
\item
Assume that we have a Whitney  decomposition ${\bf T}\mathcal{M}={\bf K}\mathcal{M}\oplus{\bf V}\mathcal{M}$ and let $p_{\bf K}$ be the associated projection on ${\bf K}\mathcal{M}$. Then, for any section ${\bf X}\in \mathfrak{P}({\bf T}\mathcal{M}_U)$, the induced section ${\bf X}^K=p_{\bf K}\circ {\bf X}$ belongs to $\mathfrak{P} \left( {\bf K}\mathcal{M}_U \right) =\Gamma \left( {\bf K}\mathcal{M}_U \right) \cap  \mathfrak{P} \left( {\bf T}\mathcal{M}_U \right) $. In particular, ${\bf X}^K$ is projectable if and only if ${\bf X}$ is so.
\end{enumerate}
\end{lemma}
\begin{proof}
(1) First of all, using the Leibniz formula, if $\mathbf{X}_1$ and $\mathbf{X}_2$ are projectable sections defined on $\mathcal{M}_U$, we have
\[
[\mathbf{X}_1,f\mathbf{X}_2]_{{\bf T}\mathcal{M}_U}=df(\hat{\rho}(\mathbf{X}_1))\mathbf{X}_2+ f[\mathbf{X}_1,\mathbf{X}_2]_{{\bf T}\mathcal{M}_U}
\]
for any $f\in C^\infty(\mathcal{M})$. Now any local section $\mathbf{X}$ of $\mathfrak{P}({\bf T}\mathcal{M}_U)$ is a finite linear functional sum
\[
\mathbf{X}=f_1 \mathbf{X}_2+\cdots+ f_k\mathbf{X}_k
\]
where, for $i \in \{ 1\dots,k \}$, each $\mathbf{X}_i$ is projectable and $f_i$ is a local smooth function on $\mathcal{M}_U$. Therefore, such a decomposition allows to define the bracket
$[\mathbf{Y},\mathbf{X}]_{{\bf T}\mathcal{M}_U}$ for all projectable sections $\mathbf{Y}$ defined on the same open set as $\mathbf{X}$. Note that, from (\ref{eq_projectTMbracket}) and the Leibniz formula, the value $[\mathbf{Y},\mathbf{X}]_{{\bf T}\mathcal{M}_U}(m)$ only depends on the $1$-jet of $\mathbf{Y}$ and $\mathbf{X}$ at point $m$ and so the value of $[\mathbf{Y}_,\mathbf{X}]_{{\bf T}\mathcal{M}_U}(m)$ is well defined. Since  such a value is independent on the expression of $\mathbf{X}$, by similar arguments, we can define the bracket $[\mathbf{X}',\mathbf{X}]_{{\bf T}\mathcal{M}_U}(m)$, for any other local section  $\mathbf{X}'$ of  $\mathfrak{P}({\bf T}\mathcal{M}_U)$. Now since, by assumption, $[.,.]_\mathcal{A}$ and the Lie bracket of vector fields satisfies the assumption of Definition \ref{D_AlmostLieBracketOnAnAnchoredBundle}, the restriction of  $[.,.]_{{\bf T}\mathcal{M}_U}$ to projectable sections satisfies also the assumption of Definition \ref{D_AlmostLieBracketOnAnAnchoredBundle} and so is its extension to sections of $\mathfrak{P}({\bf T}\mathcal{M}_U)$ for any open set $U$ of $M$ is an almost Lie bracket. \\
In this context, the   jacobiator\index{jacobiator} on ${{\bf T}\mathcal{M}_U}$ is defined by\\
$J_{{\bf T}\mathcal{M}_U}(\mathbf{X}_1,\mathbf{X}_2,\mathbf{X}_3)\hfill{}$
${}\hfill =[[\mathbf{X}_1,\mathbf{X}_2]_{{\bf T}\mathcal{M}_U},\mathbf{X}_3]_{{\bf T}\mathcal{M}_U} +[[\mathbf{X}_2,\mathbf{X}_3]_{{\bf T}\mathcal{M}_U},\mathbf{X}_1]_{{\bf T}\mathcal{M}_U}+[ [\mathbf{X}_3,\mathbf{X}_1]_{{\bf T}\mathcal{M}_U},\mathbf{X}_2]_{{\bf T}\mathcal{M}_U}$\\
But according to relation (\ref{eq_HatRhoMorphism}) for projectable sections, using the Leibnitz property, it is easy to see that, for all  $\mathbf{X}_1,\mathbf{X}_2$ in $\mathfrak{P}(\mathbf{T}\mathcal{M}_U)$
\[
\hat{\rho}([\mathbf{X}_1,\mathbf{X}_2]_{{\bf T}\mathcal{M}})=[\hat{\rho}(\mathbf{X}_1),\hat{\rho}(\mathbf{X}_2)].
\]
On the other hand, from(\ref{eq_projectTMbracket}),  for projectable sections $\mathbf{X}_i$ , $i \in \{1,2,3\}$, it follows that  
\[
J_{{\bf T}\mathcal{M}_U}(\mathbf{X}_1,\mathbf{X}_2,\mathbf{X}_3)=0.
\]
Therefore, according to these properties, it follows that $J_{{\bf T}\mathcal{M}_U}$ vanishes identically on $\mathfrak{P}({\bf T}\mathcal{M}_U)$ , which ends the proof of (1).

(2) Choose $x_0\in M$. According to the proof of Theorem \ref{T_Prolongation}, there exists a chart domain $U$ around $x$ such that
\begin{description}
\item
$\mathcal{M}_U\equiv U\times\mathbb{O}$
\item
$\mathcal{A}_U\equiv U\times\mathbb{A}$
\item
$T\mathcal{M}_U\equiv U\times\mathbb{O}\times\mathbb{M}\times\mathbb{E}$
\item
$\widetilde{\mathcal{A}}_U=U\times\mathbb{O}\times\mathbb{A}$
\item
${\bf T}\mathcal{M}_U\equiv U\times\mathbb{O}\times\mathbb{A}\times\mathbb{E}$.
\end{description}
Consider  $(a_0,v_0)\in {\bf T}_{m_0}\mathcal{M}$ where $m=(x_0,e_0)\in \mathcal{M}_U$. Using local coordinates, if $m_0\equiv \mathsf{(x_0,e_0)}$ and $(a_0,v_0)\equiv \mathsf{(a_0,v_0)}$ we consider
 the section
\[
\mathsf{X}: U\times\mathbb{O}\to U\times\mathbb{O}\times\mathbb{A}\times\mathbb{E}
\]
given by $ \mathsf{X(x,e)=(x,e, a_0, v_0)}$. Then, by construction, the corresponding local section ${\bf X}\equiv X$ is a projectable section defined on $\mathcal{M}_U$.\\

(3) Under the assumptions of (3),  note that the restriction ${\bf p}_{{\bf K}\mathcal{M}}$ of
 ${\bf p}_{\widetilde{\mathcal{A}}}$ to ${\bf K}\mathcal{M}$ is an isomorphism onto $\widetilde{\mathcal{A}}$.
  Let ${\bf X}$ be a section of ${\bf T}\mathcal{M}_U$. Thus   the difference   ${\bf X}-{\bf X}^K$  is a vertical section. Since the sum of two projectable sections is a projectable section,  it follows that ${\bf X}$ is projectable if and only if ${\bf X}^K$ is projectable. 
\end{proof}

\begin{notations}
\label{N_SheafSectionBracketlocal}
${}$
\begin{enumerate}
\item\textbf{Sheaf $\mathfrak{P}_\mathcal{M}$ of sections of $\mathbf{T}\mathcal{M}$:}\\
On the one hand,  for any open set $\mathcal{U}$ of $\mathcal{M}$, if $U=\mathbf{p}(\mathcal{U})$, then  $\mathcal{U}\subset \mathcal{M}_U$ and so 
$\mathfrak{P}(\mathcal{U}):=\{(\mathbf{X}_{| \mathcal{U}},\; \mathbf{X}\in 
\mathfrak{P}(\mathcal{M}_U)\}$.\\
On the other hand, since the set of smooth sections of a Banach bundle defines a sheaf over its basis, it follows that the set  $\{\mathfrak{P}({\bf T}\mathcal{U}),\; \mathcal{U} \textrm{ open set in }\mathcal{M}\}$ defines a sub-sheaf of modules $\mathfrak{P}_\mathcal{M}$ of the sheaf of modules  $\Gamma_\mathcal{M}$ of sections of ${\bf T}\mathcal{M}$.
 Thus  $\{\mathfrak{P}(\mathcal{M}_U),\; U\textrm{ open set in } M\}$ generates a sheaf of modules $\mathfrak{P}_\mathcal{M}$ on $\mathcal{M}$. According to Definition \ref{D_PartialConvenientLieAlgebroid}, in this 
context, when no confusion is possible, we simply denote by $[.,.]_{{\bf T}\mathcal{M}}$ the sheaf of  brackets  generated by $\{[.,.]_{{\bf T}\mathcal{M}_U}, U\textrm{ open set in  } M\}$.\\

\item
{\bf Local version of the Lie bracket $[.,.]_{{\bf T}\mathcal{M}} $:}\\
Each section ${\bf X}$ in  $\mathfrak{P}({\bf T}\mathcal{M}_U)$ has a decomposition:
\[
{\bf X}=\displaystyle\sum_{i=1}^p f_i\mathbf{X}_i
\]
with  $f_i \in C^\infty(\mathcal{M}_U)$ and ${\bf X}_i$  are  projectable sections for all $i \in \{1,\dots,p \}$.

We consider a chart domain $U$ in $M$ for which the situation considered in the proof of Lemma \ref{L_extbracket} (2) is valid. In the associated local coordinates, for any section ${\bf X}$ of ${\bf T}\mathcal{M}_U$, we have ${\bf X}(x,e)\equiv\mathsf{(x,e, a(x,e), z(x,e))}$. Now, ${\bf X}$ is projectable if and only if $\mathsf{a}$ only depends on $\mathsf{x}$. Under these notations, if $\bf{X}'\equiv \mathsf{(x,e, a'(e), z'(x,e))}$ is another projectable section, we have (cf (\ref{eq_projectTMbracket}) and Notations \ref{eq_projectTMbracket}):
\begin{eqnarray}
\label{eq_BracketProject}
\begin{aligned}
\left[{\bf X},{\bf X}' \right]_{{\bf T}\mathcal{M}}(x,e)
\equiv &\mathsf{(x,e, C_x(a,a')} + d \mathsf{a'(\rho_x(a))} - d \mathsf{a(\rho_x(a')),}\\
&d \mathsf{ z'(\rho_x(a),z)} - d \mathsf{ z((\rho_x(a'),z')))}.
\end{aligned}
\end{eqnarray}

Now consider two sections ${\bf X}$ and ${\bf X}'$ in $\mathfrak{P} \left( {\bf T}\mathcal{M}_U \right)$. We can write ${\bf X}=\displaystyle\sum_{i=1}^p f_i\mathbf{X}_i$  and   ${\bf X}'=\displaystyle\sum_{j=1}^q f'_j\mathbf{X}'_j$ where $\left( f_i, f'_i \right) \in C^\infty(\mathcal{M}_U)^2$ and ${\bf X}_i$ and ${\bf X}'_j$ are  projectable sections for all $i \in \{1,\dots,p\}$ and $j \in \{1,\dots,q\}$.  Since  the value of the  Lie bracket $[{\bf X},{\bf X}']_{{\bf T}\mathcal{M}}$  at $m$ only  depends on the $1$-jet  of ${\bf X}$ and ${\bf X}'$ and $m$, so  this value does not depend on the previous decompositions. Now ${\bf X}$ and ${\bf X}'$ can be also written as a pair $(\mathfrak{a}, X)$ and $(\mathfrak{a}', X')$ respectively. Of course, if $m=(x,e)$,  we have
\[
\mathfrak{a}(m)=  \displaystyle\sum_{i=1}^p f_i(m)a_i(x) \textrm{   and   } \mathfrak{a}'(m)=  \displaystyle\sum_{j=1}^q f'_i(m)a'_i(x)
\]
\[
X= \displaystyle\sum_{i=1}^p f_iX_i \textrm{   and   } X= \displaystyle\sum_{i=1}^p f_iX_i.
\]
In local coordinates, we then have
\[
{\bf X}\equiv \mathsf{(a,z)=\left( \displaystyle\sum_{i=1}^p f_i(m)a_i,\sum_{i=1}^p f_iz_i \right) }
\]
\[
{\bf X}'\equiv \mathsf{(a',z')= \left( \displaystyle\sum_{j=1}^q f'_ja'_j,\sum_{j=1}^q f_jz_j \right) }.
\]
Thus the Lie bracket $[{\bf X},{\bf X}']_{{\bf T}\mathcal{M}}$ has the following expression in local coordinates:
\begin{eqnarray}
\label{eq_BracketQuasiprojectable}
\begin{aligned}
\left[{\bf X},{\bf X}'\right]_{{\bf T}\mathcal{M}}&
\equiv \mathsf{( C(a,a') } +d \mathsf{a'(\rho_x(a'))}\\
&- d \mathsf{a(\rho_x(a')),} d \mathsf{z'(\rho_x(a),z)} -d \mathsf{z((\rho_x(a'),z'))}
\end{aligned}
\end{eqnarray}
where $\mathsf{C}: U\to {L}_{alt}^2(\mathbb{A})$ is defined in local coordinates by the Lie bracket $[.,.]_\mathcal{A}$ and where
\[
\mathsf{C(a,a')(m)= \displaystyle\sum_{i=1}^p \sum_{j=1}^q f_i(m)f'_j(m) C_x(a_i(x),a'_j(x))}.
\]
\end{enumerate}
\end{notations}

\begin{remark}
\label{R_AFiniterank}
According to Comments \ref{Com_PartialBracket}, when $\mathcal{A}$ is a finite rank vector bundle, then $\mathfrak{P}(\mathcal{M}_U=\Gamma(\mathbf{T}\mathcal{M}_U$ and so $[.,.]_{\mathbf{T}\mathcal{M}_U}$ is defined for  all sections of $\Gamma(\mathbf{T}\mathcal{M}_U)$.
\end{remark}
Now, from  Lemma \ref{L_extbracket},  $ \left( \mathfrak{P}({\bf T}\mathcal{M}_U), [.,.]_{{\bf T}\mathcal{M}} \right)$ is a Lie algebra and $\hat{\rho}$ induces a Lie algebra morphism from  $(\mathfrak{P}({\bf T}\mathcal{M}_U), [.,.]_{{\bf T}\mathcal{M}})$ to  the  Lie algebra of vector fields on $\mathcal{M}_U$. In this way,  so we get:

\begin{theorem}
\label{T_PartialLieAlgebroid} 
The  sheaf $\mathfrak{P}_\mathcal{M}$ on $\mathcal{M}$  which gives rise to a strong partial convenient Lie algebroid on the anchored bundle
$({\bf T}\mathcal{M}, \hat{p}, \mathcal{M}, \hat{\rho})$. Moreover, the restriction of the bracket $[.,.]_{{\bf T}\mathcal{M}}$ to
the module of vertical sections induces a  convenient Lie algebroid  structure on the anchored subbundle $({\bf V}\mathcal{M}, \hat{p}_{|{\bf V}\mathcal{M}},\mathcal{M},\hat{\rho})$ which is independent on the bracket $[.,.]_\mathcal{A}$.
\end{theorem}
From Remark \ref{R_AFiniterank}, we obtain:
\begin{corollary}
\label{C_LieAlgebroidTM}
If $\mathcal{A}$ is a finite dimensional fiber Banach bundle, then\\ 
$ \left( {\bf T}\mathcal{M},\hat{p}, \mathcal{M}, \hat{\rho},[.,.]_{{\bf T}\mathcal{M}} \right) $ is a convenient Lie algebroid.
\end{corollary}

\begin{proof}[Proof of Theorem \ref{T_PartialLieAlgebroid}]
From  Lemma \ref{L_extbracket},   $(\mathfrak{P}({\bf T}\mathcal{M}_U), [.,.]_{{\bf T}\mathcal{M}})$ is a Lie algebra and $\hat{\rho}$ induces a Lie algebra morphism from  $(\mathfrak{P}({\bf T}\mathcal{M}_U), [.,.]_{{\bf T}\mathcal{M}})$ to  the  Lie algebra of vector fields on $\mathcal{M}_U$. This implies the same properties for  $(\mathfrak{P}({\bf T}\mathcal{U}), [.,.]_{{\bf T}\mathcal{M}})$ for any open set  $\mathcal{U}$ in $\mathcal{M}$. Thus we obtain a sheaf $\mathfrak{P}_\mathcal{M}$ of Lie algebras and $C^\infty_\mathcal{M}$ modules on $\mathcal{M}$, which  implies that   $({\bf T}\mathcal{M},\hat{p}, \mathcal{M}, \hat{\rho},[.,.]_{{\bf T}\mathcal{M}})$  is a  partial convenient Lie algebroid  which is strong, according to Lemma \ref{L_extbracket} (2).\\

It remains to prove the last property. From its  definition,  the restriction of $\hat{\rho}$ to the vertical bundle ${\bf V}\mathcal{M}\to \mathcal{M}$ is an isomorphism onto the vertical bundle of the tangent bundle $T\mathcal{M}\to \mathcal{M}$. Now, from the definition  of the bracket of  projectable sections, it is clear that the bracket of two (local) vertical sections of ${\bf T}\mathcal{M}\to \mathcal{M}$ is also a (local) vertical section which is independent on the choice of $[.,.]_\mathcal{A}$.\\
\end{proof}
\begin{remark}
\label{R_liftmorphism}
Given a fibred  morphism  $\Psi:{\mathcal{M}}\to {\mathcal{M}}'$ between fibred manifold ${\bf p}:{\mathcal{M}}\to M$  and ${\bf p}':{\mathcal{M}'}\to M'$  over $\psi:M\to M'$ and a morphism of anchored bundles  $ \varphi$ between two Lie algebroids $({\mathcal{A}}, \pi, M,\rho, [.,.]_{\mathcal{A}})$ and $({\mathcal{A}}', \pi',M',\rho',[.,.]_{\mathcal{A}'})$,  over  $\psi:M\to M'$. Then,  from Remark \ref{R_LieMorphismPartialAlgebroid},  the prolongation ${\bf T}\Psi: {{\bf T}}^{\mathcal{A}}{\mathcal{M}}\to {{\bf T}}^{\mathcal{A}'}{\mathcal{M}}'$,  is a Lie morphism of  partial convenient Lie algebroids from $({\bf T}^{\mathcal{A}}\mathcal{M},\mathcal{M},\hat{\rho},\mathfrak{P}_{\mathcal{M}})$ to $({\bf T}^{\mathcal{A}'}\mathcal{M}',\mathcal{M}',\hat{\rho}',\mathfrak{P}_{\mathcal{M}'})$.
\end{remark}

\subsection{Derivative operator on ${\bf T}\mathcal{M}$}
\label{___LieDerivativeExteriorDifferentialAndNijenhuisTensorOnTM}

Consider an open set $\mathcal{U}$ in $\mathcal{M}$. We simply denote ${\bf T}\mathcal{U}$ the restriction of ${\bf T}\mathcal{M}$ to $\mathcal{U}$. Note that $U={\bf p}(\mathcal{U})$ is an open set in $M$ and of course $\mathcal{U}$ is contained in the open set $\mathcal{M}_U$; so we have a natural restriction map from $\Gamma({\bf T}\mathcal{M}_U)$ (resp. $\mathfrak{P}({\bf T}\mathcal{M}_U)$) into $\Gamma({\bf T}\mathcal{U})$ (resp. $\mathfrak{P}({\bf T}\mathcal{U})$).\\

Since   ${\bf T}\mathcal{M}$ has only a  strong partial convenient Lie algeboid structure, by application of results of  subsection \ref{___DerivativeOperators}, we then have the  following result:
\begin{theorem}
\label{T_DifferentialOperatorTM}
Fix some open set $\mathcal{U}$ in $\mathcal{M}$.
\begin{enumerate}
\item 
Fix some projectable section $\mathfrak{u}\in \mathfrak{P}({\bf T}\mathcal{U})$. For any $k$-form $\omega$ on $\mathcal{U}$  the Lie derivative
$L^{\hat{\rho}}_\mathfrak{ u}\omega$  is a well defined $k$-form on $\mathcal{U}$.
\item 
For any $k$-form $\omega$ on $\mathbf{T}\mathcal{U}$ the exterior derivative $d_{\hat{\rho}}\omega$ of $\omega$ is well defined $(k+1)$-form on $\mathcal{U}$.\\
\end{enumerate}
\end{theorem}

\subsection{Prolongations and foliations}
\label{___ProlongationsAndFoliations}
\emph{We assume that $(\mathcal{A},M,\rho,[.,.]_\mathcal{A})$ is a split  Banach Lie algebroid}.\\

Under the upper assumptions, by application of Theorem \ref{T_Prolongation} and  Theorem 2 in \cite{Pe1}, we obtain  the following  link between the foliation on $M$ and the foliation on $T\mathcal{M}$:
\begin{theorem}
\label{T_tildefol}
Assume that $(\mathcal{A},\pi,M,\rho,[.,.]_\mathcal{A})$ is split  and the distribution $\rho(\mathcal{A})$ is closed.
\begin{enumerate}
\item
The distribution $\hat{\rho}({\bf T}\mathcal{A})$ is integrable on ${\bf T}\mathcal{M}$.
\item
Assume that $\mathcal{M}=\mathcal{A}$. Let $L$  be a leaf of $\rho(\mathcal{A})$ and $(\mathcal{A}_L,\pi_{|L},L,\rho_L,[.,.]_{\mathcal{A}_L})$ be the Banach-Lie algebroid  which is the restriction of $(\mathcal{A}, \pi, M,\rho,[.,.]_{\mathcal{A}})$. Then $\hat{L}=\mathcal{A}_L=\pi^{-1}(L)$ is a leaf of $\hat{\rho}( \mathbf{T}\mathcal{A})$ such that $\hat{\mathbf{p}}(\hat{L})=L$.
\item
In the previous context, the partial Banach-Lie algebroid which is the  prolongation of  $(\mathcal{A}_L,\pi_{|L},L,\rho_L,\mathfrak{P}_{\mathcal{A}_L})$ over $\hat{L}$  is exactly the restriction to $\hat{L}$ of the partial  Banach-Lie algebroid $({\bf T}\mathcal{A}, \hat{p},\mathcal{A},\hat{\rho}, \mathfrak{P}_{\mathcal{A}})$.
\end{enumerate}
\end{theorem}
\begin{remark}
\label{R_prolongAL}
${}$
\begin{enumerate}
\item
If there exists a weak Riemannian metric on $\mathcal{A}$, the distribution $\rho(\mathcal{A})$ is closed and the assumptions of Theorem \ref{T_tildefol} are satisfied. These assumptions are always satisfied if $\rho$ is a Fredholm morphism.
\item
The set of leaves of the foliations defined by $\hat{\rho}(\mathbf{T}\mathcal{A})$ is
\[
\{\mathcal{A}_L, \; L \textrm{ leaf of } \rho(\mathcal{A})\;\}.
\]
\end{enumerate}
\end{remark}
\begin{proof}
From Theorem \ref{T_Prolongation}, if  $m=(x,e)\in \mathcal{A}$,  the fibre  ${\bf T}_m\mathcal{M}$ can be identified with ${\mathcal{A}}_x\times \widetilde{\mathcal{E}}_m\equiv \mathbb{A}\times\mathbb{E}$, so we have $\hat{\rho}_m(a, v)\equiv (\mathsf{\rho_x(a),v})$.  It follows that $\ker(\hat{\rho}_m)$ can be identified with $\ker\rho_x\times{\bf V}_m\mathcal{M}\subset {\mathcal{A}}_x\times T_m\mathcal{A}$. Thus, $\ker\hat{\rho}_m$ is split if and only if $\ker\rho_x$ is split.
Moreover $\hat{\rho}_m({\bf T}_m\mathcal{A})=\rho_x(\mathcal{A}_x)\times {\bf V}_m\mathcal{M}$  is closed in ${\bf T}_m\mathcal{M}\equiv\mathbb{M}\times\mathbb{E}$ if and only if $\rho_x(\mathcal{A}_x)$ is closed in $T_xM$.  Then (1) will be a consequence of \cite{BCP}, Theorem 8.39 if we show that, for any $(x,e)\in\mathcal{M}$, there exists an open neighbourhood $U$ of $x$ in $M$  such that the set $\mathfrak{P}({\bf T}\mathcal{M}_U)$ is a generating upper set for $\hat{\rho}({\bf T}\mathcal{M})$ around $(x,e)$ (cf.  \cite{BCP}, Definition 8.38) and satisfied the condition {\bf (LB)} given in \cite{BCP}, $\S$ Integrability and Lie invariance.\\

The point $m=(x,e)\in \mathcal{M}$ is fixed and  we  choose a chart domain $U$ of $x\in M$ such that $\mathcal{M}_U$ and $\mathcal{A}_U$ are trivializable. Without loss of generality, according to the notations in $\S$ \ref{___LocalIdentificationsAndExpressionsInAConvenientBundle}, we may assume that $U\subset \mathbb{M}$, $\mathcal{M}_U=U\times\mathbb{O}\subset \mathbb{M}\times\mathbb{E}$, $\mathcal{A}_U=U\times\mathbb{A}$.
Thus, according to the proof of Theorem \ref{T_Prolongation}, we have ${\bf T}\mathcal{M}_U=U\times\mathbb{O}\times \mathbb{A}\times\mathbb{E}$. In this context, if $\{\mathsf{x}\}\times\mathbb{A}=\ker \rho_\mathsf{x}\oplus \mathbb{S}$,
then
\[
\{\mathsf{(x,e)}\}\times \mathbb{A}\times\mathbb{E}=\{\mathsf{(x,e)}\}\times(\ker \rho_\mathsf{x}\oplus \mathbb{S})\times \mathbb{E}.
\]
Now consider the upper trivialization $\rho: U\times \mathbb{A}\to U\times\mathbb{M}(=TM_U)$ and the associated upper trivialization:
\[
\hat{\rho}:U\times\mathbb{O}\times\mathbb{A}\times\mathbb{E}\to U\times\mathbb{O}\times\mathbb{A}\times\mathbb{M}(=T\mathcal{M}_U).
\]
Then, from the definition of $\hat{\rho}$,  any upper vector field is  of type
\[
{\bf X}_{\mathsf{(a,v)}}=\hat{\rho}\mathsf{(a, v)= (\rho(a), v)}
\]
for any $\mathsf{(a, v)}\in \mathbb{A}\times\mathbb{E}$. From the proof of Lemma \ref{L_extbracket} (2),  it follows that such a vector field is the range, under $\hat{\rho}$, of a projectable section. Moreover, the stability of projectable sections under the Lie bracket $[.,.]_{{\bf T}\mathcal{M}_U}$ and the fact that $\hat{\rho}$ induces a Lie algebra morphism from $\mathfrak{P}({\bf T}\mathcal{M}_U)$ into the Lie algebra of vector fields on $\mathcal{M}_U$  implies that condition ({\bf LB})  is satisfied for the set $\mathfrak{P}({\bf T}\mathcal{M}_U)$.\\
Assume that $\mathcal{M}=\mathcal{A}$. Fix some leaf $L$ in $M$. If $(\mathcal{A}_L,\pi_{|L},L,\rho_L,[.,.]_{\mathcal{A}_L})$ is  the Banach-Lie algebroid  which is the restriction of $(\mathcal{A}, \pi, M,\rho,[.,.]_\mathcal{A})$  then $\pi^{-1}(L)=\mathcal{A}_L$ and $\rho_L(\mathcal{A}_L)=TL$. Now by construction, the prolongation $\mathbf{T}\mathcal{A}_L$ of $\mathcal{A}_L$ relative to  the Banach-Lie algebroid  $(\mathcal{A}_L,\pi_{|L},L,\rho_L,[.,.]_{\mathcal{A}_L})$ is characterized by
\[
\mathbf{T}_{(x,a)}\mathcal{A}_L=\{(b, (v,y))\in \mathcal{A}_x \times T_{(x,a)}\mathcal{A}_L\;:\; \rho_x(b)=y\}.
\]
It follows that $\mathbf{T}\mathcal{A}_L$ is the restriction of $\mathbf{T}\mathcal{A}$ to $\mathcal{A}_L$ and also $\hat{\rho}(\mathbf{T}\mathcal{A}_L)=T\mathcal{A}_L$. Since $L$ is connected, so is $\mathcal{A}_L$ and then $\mathcal{A}_L$ is a leaf of $\hat{\rho}(\mathbf{T}\mathcal{M})$.
\end{proof}

\section{Projective limits of prolongations of Banach Lie algebroids}
\label{__ProjectiveLimitsOfProlongationsOfBanachAnchoredBundles}

\begin{definition}
\label{D_ProjectiveSequenceofBanachLieAlgebroids}
Consider a projective sequence of Banach-Lie algebroid bundles
\begin{center}
 $\left( \left(  \mathcal{A}_{i},\pi_{i},M_{i},\rho_{i},[.,.]_i \right),\left( \zeta_i^j, \xi_i^j, \delta_i^j \right) \right) _{(i,j)\in\mathbb{N}^2, j \geq i}$
\end{center}
(resp. of  Banach bundles  $\left( \left(  \mathcal{E}_{i},\mathbf{\pi}_{i},M_{i},\rho_{i},[.,.]_i \right),\left( \xi_{i}^{j}, \delta_{i}^{j} \right) \right) _{(i,j)\in\mathbb{N}^2, j \geq i}$).\\
A sequence of open fibred manifolds $\mathcal{M}_i $ of $\mathcal{E}_i$ is called compatible with  algebroid prolongations, if,  for all $\left( i,j \right) \in\mathbb{N}^2$ such that $j\geq i$, we have
\begin{description}
\item[\textbf{(PSPBLAB 1)}]
{\hfil $\xi_i^j(\mathcal{M}_j)\subset \mathcal{M}_i$;}
\item[\textbf{(PSPBLAB 2)}]
{\hfil $\mathbf{p}_i \circ \xi_i^j = \delta_i^j \circ \mathbf{p}_j$.}
\end{description}
\end{definition}

Under the assumptions of Definition \ref{D_ProjectiveSequenceofBanachLieAlgebroids}, for each $i\in \mathbb{N}$, we denote by  $(\mathbf{T}\mathcal{M}_i, \hat{\mathbf{p}}_i, \mathcal{M}_i, \hat{\rho}_i)$ the prolongation  of $\mathcal{A}_i$ over $\mathcal{M}_i$ and $[.,.]_{\bf{T}\mathcal{M}_i}$ the prolongation of the Lie bracket $[.,.]_{\mathcal{A}_i}$ on projectable sections of $\mathbf{T}\mathcal{M}_i$.

We then have the following result.

\begin{proposition}
\label{P_ProjectiveLimitProlongationBracket}
Consider a projective sequence of Banach-Lie algebroid  bundles  
\begin{center}
$\left( \left(  \mathcal{A}_{i},\pi_{i},M_{i},\rho_{i},[.,.]_{\mathcal{A}_i} \right),\left( \zeta_i^j, \xi_i^j, \delta_i^j \right) \right) _{(i,j)\in\mathbb{N}^2, j \geq i}$ 
\end{center}
(resp. Banach bundles $\left( \left(  \mathcal{E}_{i},\mathbf{\pi}_{i},M_{i},\rho_{i},[.,.]_i \right),\left( \xi_{i}^{j}, \delta_{i}^{j} \right) \right) _{(i,j)\in\mathbb{N}^2, j \geq i}$) and a sequence of open fibred manifolds $\mathcal{M}_i $ of $\mathcal{E}_i$ compatible with  algebroid prolongations. Then
\begin{enumerate}
\item 
$\left( \mathbf{T}\mathcal{M}=\underleftarrow{\lim}\mathbf{T}\mathcal{M}_i,\hat{\mathbf{p}}=\underleftarrow{\lim}\hat{\mathbf{p}}_i,M=\underleftarrow{\lim}\mathcal{M}_i,
\hat{\rho}=\underleftarrow{\lim}\hat{\rho}_i \right)  $ is a Fr\'{e}chet anchored bundle which is the prolongation of
$\left( \mathcal{A}=\underleftarrow{\lim}\mathcal{A}_i,\pi=\underleftarrow{\lim}\pi_i,M=\underleftarrow{\lim}M_i \right)$ over $M$.\\
\item 
Consider any of open $U$ in $M$  and a sequence  of open set $U_i$ in $M_i$ such that $U=\underleftarrow{\lim}U_i$.  We denote by $\mathfrak{P}^{pl}(\mathbf{T}\mathcal{M}_U)$ the $C^\infty(U)$-module generated by all projective limits $\mathbf{X}=\underleftarrow{\lim}\mathbf{X}_i$  of projectable sections  $\mathbf{X}_i$  of $\mathbf{T}\mathcal{M}_i$ over $\{\mathcal{M}_i\}_{U_i}$
 Then there exists a Lie bracket
  $[.,.]_{\mathbf{T}\mathcal{M}} $ defined on  $\mathfrak{P}^{pl}(\mathbf{T}\mathcal{M}_U)$ which  satisfies the assumptions of Definition \ref{D_AlmostLieBracketOnAnAnchoredBundle} characterized by
\[
[\mathbf{X}, \mathbf{X}']_{\mathbf{T}\mathcal{M}_U}=\underleftarrow{\lim}[\mathbf{X}_i,\mathbf{X}'_i]_{\mathbf{T}\mathcal{M}_i}
\]
where  $\mathbf{X}=\underleftarrow{\lim}\mathbf{X}_i$ and  $\mathbf{X}'=\underleftarrow{\lim}\mathbf{X}'_i$.\\
\end{enumerate}
\end{proposition}

Note that $\mathfrak{P}^{pl}(\mathbf{T}\mathcal{M}_U)$ is a submodule of the module $\mathfrak{P}(\mathbf{T}\mathcal{M}_U)$ generated by projectable sections $\mathbf{T}\mathcal{M}_U$. Therefore, by analog 
argument as used in the proof of Theorem \ref{T_PartialLieAlgebroid}, the set $\{\mathfrak{P}^{pl}(\mathbf{T}\mathcal{M}_U), U \textrm{ open set  in } M\}$ generates a sheaf of $\mathfrak{P}^{pl}_\mathcal{M}$ over $
\mathcal{M}$. Moreover,   for any open set $\mathcal{U}$ in $ \mathcal{M}$, according to Proposition \ref{P_ProjectiveLimitProlongationBracket}, the restriction of $\hat{\rho}$ to each $\mathfrak{P}^{pl}(\mathbf{T}\mathcal{U})$ is a Lie algebra morphism into the  Lie algebra morphism into the Lie algebra of vector fields on $\mathcal{U}$.
Thus we obtain:

\begin{theorem}
\label{T_ProjectiveLimitOfProlongationOfBanachAnchoredBundles}
Consider a projective sequence of Banach-Lie algebroid  bundles  
\begin{center}
$\left( \left(  \mathcal{A}_{i},\pi_{i},M_{i},\rho_{i},[.,.]_{\mathcal{A}_i} \right),\left( \zeta_i^j, \xi_i^j, \delta_i^j \right) \right) _{(i,j)\in\mathbb{N}^2, j \geq i}$
\end{center}
(resp. Banach bundles $\left( \left(  \mathcal{E}_{i},\mathbf{\pi}_{i},M_{i},\rho_{i},[.,.]_i \right),\left( \xi_{i}^{j}, \delta_{i}^{j} \right) \right) _{(i,j)\in\mathbb{N}^2, j \geq i}$) and a sequence of open fibred manifolds $\mathcal{M}_i $ of $\mathcal{E}_i$  compatible with  algebroid prolongations. Then $\left( \mathbf{T}\mathcal{M}=\underleftarrow{\lim}\mathbf{T}\mathcal{M}_i,\hat{\mathbf{p}}=\underleftarrow{\lim}\hat{\mathbf{p}}_i,M=\underleftarrow{\lim}\mathcal{M}_i,
\hat{\rho}=\underleftarrow{\lim}\hat{\rho}_i \right)  $ is a Fr\'{e}chet anchored bundle which is the prolongation of
$\left( \mathcal{A}=\underleftarrow{\lim}\mathcal{A}_i,\pi=\underleftarrow{\lim}\pi_i,M=\underleftarrow{\lim}M_i \right)$ over $M$. Moreover $\left(\mathbf{T}\mathcal{M},\hat{\bf p}, \mathcal{M}, \hat{\rho},\mathfrak{P}^{pl}_{\mathcal{M}}\right)$ is a strong  partial Fr\'echet Lie algebroid.
\end{theorem}

\begin{remark}
\label{R_NotPartialLieAlgebroidProlongation} 
In general, since the projective limit of a projective sequence of Banach algebroids has only a structure of  partial Fr\'echet Lie algebroid, it follows that $\mathfrak{P}^{pl}_\mathcal{M}$  is a subsheaf of modules of $\mathfrak{P}_\mathcal{M}$ and the inclusion is strict in the sense that for each open  $\mathcal{U}$ in $\mathcal{M}$, the inclusion of $\mathfrak{P}^{pl}(\mathcal{U})$ in $\mathfrak{P}(\mathcal{U})$ is strict. Thus we do not have a structure of strong partial Fr\'echet Lie algebroid defined on  $\mathfrak{P}_\mathcal{M}$ as in Theorem \ref{T_PartialLieAlgebroid}.
\end{remark}

\begin{proof}[Proof of Proposition \ref{P_ProjectiveLimitProlongationBracket}]${}$\\
(1) According to \textbf{(PSPBLAB 1)} and Theorem \ref{T_ProjectiveLimitOfBanachLieAlgebroids},
$\left( \underleftarrow{\lim}\mathcal{A}_i,\underleftarrow{\lim}\pi_i,\underleftarrow{\lim}M_i ,\underleftarrow{\lim}\rho_i \right)$ is a Fr\'echet anchored bundle. \\
From \textbf{(PSPBLAB 2)} and Proposition \ref{P_ProjectiveLimitOfBanachVectorBundles}, we obtain a structure of Fr\'echet vector bundle on $\left( \underleftarrow{\lim}\mathcal{E}_i,\underleftarrow{\lim}\mathbf{p}_i,\underleftarrow{\lim}M_i  \right) $. Since each $\mathcal{M}_i$ is an open manifold of $\mathcal{E}_i$ such that the restriction of $\mathbf{p}_i$ is a surjective fibration of $\mathcal{M}_i$ over $M_i$ and we have $\xi_i^j(\mathcal{M}_j) \subset \mathcal{M}_i$, it follows that $(\mathcal{M}_i,\xi_i^j) _{(i,j)\in\mathbb{N}^2, j \geq i}$ is a projective sequence of Banach manifolds and so the restriction of $\mathbf{p}=\underleftarrow{\lim}\mathbf{p}_i$ to $\mathcal{M}=\underleftarrow{\lim}\mathcal{M}_i$ is a surjective fibration onto $M$\\

Recall that
\[
{ \bf T}\mathcal{M}_j
=\{ \left( a_j,\mu_j \right) \in \mathcal{A}_{x_j} \times T_{m_j}\mathcal{M}_j: \;
\rho_j \left( a_j \right) = T_{m_j}{\bf p}_j \left( \mu_j \right)
\}.
\]
Let $\left( a_j,\mu_j \right)$ be in ${ \bf T}{\mathcal{M}_j}$ and consider $a_i=\xi_{i}^{j} \left( a_j \right) $ and $\mu_i=T\xi_i^j\left( \mu_j \right) $. We then have:
\[
\rho_i\left( a_i \right) = \rho_i \circ \xi_{i}^{j} \left( a_j \right) = T\delta_i^j \circ \rho_j \left( a_j \right)
\]
and also
\[
T_{m_i}{\bf p}_i \left( \mu_i \right) = T_{m_i}{\bf p}_i \circ T\xi_i^j \left( \mu_j \right)
= T\delta_i^j \circ T_{m_j}{\bf p}_j \left( \mu_j \right).
\]
Since $\rho_j \left( a_j \right) = T_{m_j}{\bf p}_j \left( \mu_j \right)$, we then obtain $\rho_i \left( a_i \right) = T_{m_i}{\bf p}_i \left( \mu_i \right)$.

So $\mathbf{T}\xi_i^j : \mathbf{T} \mathcal{M}_j \to \mathbf{T} \mathcal{M}_i$ is a morphism of Banach bundles and we have the following commutative diagram
\[
\xymatrix{
           \mathbf{T} \mathcal{M}_i  \ar@{<-}[r]^{\mathbf{T}\xi_i^j}\ar[d]_{\hat{\rho}_i} & \mathbf{T} \mathcal{M}_j \ar[d]^{\hat{\rho}_j}\\
           T\mathcal{M}_i \ar@{<-}[r]^{T\xi_i^j}              & T\mathcal{M}_j\\
}
\]
We deduce that $\left( \left(  {\mathbf{T}}^{\mathcal{A}_i} \mathcal{M}_i,\hat{\mathbf{p}}_i,\mathcal{M}_i,\hat{\rho}_i \right),\left( \mathbf{T}\xi_i^j,\xi_{i}^{j}\right) \right) _{(i,j)\in\mathbb{N}^2, j \geq i}$ is a projective sequence of Banach anchored bundles. Applying again Theorem \ref{T_ProjectiveLimitOfBanachLieAlgebroids}, we get a Fr\'echet anchored bundle structure on $\left( \underleftarrow{\lim}\bf{T}\mathcal{M}_i,\underleftarrow{\lim}\hat{\bf{p}}_i,\underleftarrow{\lim}\mathcal{M}_i \right)$ over $\underleftarrow{\lim}\mathcal{M}_i$ and which  appears as the prolongation of $\left( \underleftarrow{\lim}\mathcal{A}_i,\underleftarrow{\lim}\pi_i,\underleftarrow{\lim}M_i ,\underleftarrow{\lim}\rho_i \right) $ over $\underleftarrow{\lim}\mathcal{M}_i $.\\

(2) Let  $U$ be an open set in $M$. There exists $U_i$ in $M_i$ such that $\delta_i(U)\subset U_i$ for each $i\in \mathbb{N}$ and so  that $U=\underleftarrow{\lim}U_i$. 
Now, from  the definition of $\{\mathcal{M}_i\}_{U_i}$, we must have $\mathcal{M}_U=\underleftarrow{\lim}\{\mathcal{M}_i\}_{U_i}$.\\
 Recall that a projectable section $\mathbf{X}_i$ on $\{\mathcal{M}_i\}_{U_i}$ is characterized by a pair $(\mathfrak{a}_i, X_i)$ where $\mathfrak{a}_i$  is a section of $\{\mathcal{A}_i\}_{U_i}$ and $X_i$ is a vector field on $\{\mathcal{M}_i\}_{U_i}$ such that $\rho_i\circ \mathfrak{a}_i=T{\mathbf{p}_i}(X_i)$.
  Assume that  $\mathfrak{a}=\underleftarrow{\lim}\mathfrak{a}_i$ and $X=\underleftarrow{\lim}X_i$, from the compatibility with bonding maps for sequences  of  sections $(\mathfrak{a}_i)$, anchors $(\rho_i)$ and vector fields $(X_i)$  we  must have then have  $T{\bf p}(X)=\rho\circ \mathfrak{a}$. But 
  from Theorem \ref{T_Prolongation}, it follows that 
 $(\mathfrak{a}\circ \mathbf{p}, X)$ defines a  projectable section $\mathbf{X}$ over  $\mathcal{M}_U$ and so $\mathbf{X}=\underleftarrow{\lim}\mathbf{X}_i$.\\

On the other hand, recall that from (\ref{eq_projectTMbracket}), for each $i\in \mathbb{N}$, we have 
\begin{eqnarray}
\label{Eq_BracketTMi}
[\mathbf{X}_i,\mathbf{X}'_i]_{{\bf T}\mathcal{M}_i}=([\mathfrak{a}_i,\mathfrak{a}'_i]_{\mathcal{A}_i}\circ {\bf p},[X_i,X'_i])
\end{eqnarray}

Now since $\xi_i^j$ is a Lie algebroid morphism, over  $\delta_i^j$, according to Definition \ref{D_LieMorphism} we have 
\begin{eqnarray}
\label{Eq_bracketfij}
\xi_i^j([\mathfrak{a}_j,\mathfrak{a}'_j)]_{\mathcal{A}_j})(x_j)=\left([\mathfrak{a}_i,\mathfrak{a}_i]_{\mathcal{A}_i}\right)(\delta_i^j(x_j)).
\end{eqnarray}
Since  $\delta_i^j\circ\mathbf{p}_j=\mathbf{p}_i\circ \lambda_i^j$, we have: 
\begin{eqnarray}
\label{Eq_bracketfij}
\xi_i^j([\mathfrak{a}_j,\mathfrak{a}'_j]_{\mathcal{A}_j})\circ\mathbf{p}_j(m_j)=\left([(\mathfrak{a}_i,\mathfrak{a}'_i]_{\mathcal{A}_i}\right)\circ \mathbf{p}_i\circ \xi_i^j(m_j).
\end{eqnarray}

Naturally, since $X_i$ (resp. $X'_i$) and $X_j$ (resp. $X'_j$) are $\xi_i^j$ related, we also have
\begin{eqnarray}\label{Eq_bracketxiij}
[X_i,X'_i)]\left(\xi_i^j(m_j)\right)=T\xi_i^j\left([X_j,X_j]\right)(m_j).
\end{eqnarray}
From  (\ref{Eq_BracketTMi})  and (\ref{eq_bfTPsi}) we then obtain
\begin{eqnarray}\label{Eq_BracketbfTM}
\begin{aligned}
\mathbf{T}\xi_i^j\left([\mathbf{X}_i,\mathbf{X}^\prime_j]_{\mathbf{T}\mathcal{M}_i}\right)(m_j)
&=\left(\xi_i^j([\mathfrak{a}_j,\mathfrak{a}^\prime_j]_{\mathcal{A}_j})\circ\mathbf{p}_j(m_j),T_{m_j}\xi_i^j([{X}_i,{X}^\prime_j])\right)\\
&=\left([\mathfrak{a}_i,\mathfrak{a}^\prime_i]_{\mathcal{A}_i}\circ\mathbf{p}_i\circ \xi_i^j(m_j),[{X}_i,X^\prime_i])\circ \xi_i^j(m_j)\right)\\ 
&=\left([\mathbf{X}_i,\mathbf{X}_i^\prime]_{\mathbf{T}\mathcal{M}_i}\right)\circ\xi_i^j(m_j).
\end{aligned}
\end{eqnarray}

It follows that we can define:
\begin{eqnarray}\label{Eq_ProjectiveLimit bracketTMi}
[\mathbf{X},\mathbf{X}']_{\mathbf{T}\mathcal{M}_U}=\underleftarrow{\lim}[\mathbf{X}_i,\mathbf{X}'_i]_{\mathbf{T}\mathcal{M}_i}.
\end{eqnarray}

Now, since each bracket $[.,.]_{\mathbf{T}\mathcal{M}_i}$ satisfies the Jacobi identity, from $(\ref{Eq_ProjectiveLimit bracketTMi})$, the same is true for $[.,.]_{\mathbf{T}\mathcal{M}}$ on projective limit $\mathbf{X}$ and $\mathbf{X'}$ of projective sections $(\mathbf{X}_i)$ and $(\mathbf{X}'_i)$.
Finally, as 
\begin{eqnarray}\label{Eq_compatibilityrho}
[\hat{\rho}_i(\mathbf{X}_i),\hat{\rho}_i\mathbf{X}'_i)]_{\mathbf{T}\mathcal{M}_i}=\hat{\rho}_i\left([\mathbf{X}_i,\mathbf{X}_i]_{\mathbf{T}\mathcal{M}_i}\right).
\end{eqnarray}
from  the compatibility with bonding maps for sequences  of  sections $(\mathfrak{a}_i)$, anchors $(\rho_i)$  vector fields $(X_i)$  and  Lie brackets $[.,.]_{\mathbf{T}\mathcal{M}_i}$ on projective sequence $(\mathbf{T}\mathcal{M}_i)$, it follows that $\hat{\rho}$ satisfies the  same type of relation as (\ref{Eq_compatibilityrho}).\\

Now using by same arguments as used in the proof of Lemma \ref{L_extbracket} we show that we can extend this bracket to a Lie bracket on the module $\mathfrak{P}^{pl}(\mathbf{T}\mathcal{M}_U)$ so that we have a Lie algebra and the restriction of $\hat{\rho}$ to this Lie algebra is a morphism of Lie algebra into the Lie algebra of vector fields on $\mathcal{M}_U$.
\end{proof}

\section{Direct limits of prolongations of Banach Lie algebroids}
\label{__DirectLimitsOfProlongationsOfBanachAnchoredBundles}

As in the previous section we introduce:

\begin{definition}
\label{D_DirectSequenceofBanachLieAlgebroids}
Consider a direct  sequence of Banach-Lie algebroid  bundles
\begin{center}
$\left( \left(  \mathcal{A}_{i},\pi_{i},M_{i},\rho_{i} ,[.,.]_{\mathcal{A}_i}\right),\left( \eta_i^j, \chi_i^j, \varepsilon_i^j \right) \right) _{(i,j)\in\mathbb{N}^2, i \leq j}$
\end{center} 
(resp. of  Banach bundles  $\left( \left(  \mathcal{E}_{i},\mathbf{\pi}_{i},M_{i} \right),\left( \chi_i^j, \delta_i^j \right) \right) _{(i,j)\in\mathbb{N}^2, i\leq j}$).\\
A sequence of open fibred manifolds $\mathcal{M}_i $ of $\mathcal{E}_i$ is called compatible with algebroid prolongations, if  for all $\left( i,j \right) \in\mathbb{N}^2$ such that $i\leq j$, we have
\begin{description}
\item[\textbf{(DSPBLAB 1)}]
{\hfil $\chi_i^j(\mathcal{M}_j)\subset \mathcal{M}_i$;}
\item[\textbf{(DSPBLAB 2)}]
{\hfil $\varepsilon_i^j \circ \mathbf{p}_i  =  \mathbf{p}_j \circ \chi_i^j$}.
\end{description}
\end{definition}

\begin{remark} 
The context of direct limit in which we work concerns ascending sequences of Banach manifolds $(M_i)_{i\in \mathbb{N}}$ where $M_i$ is a closed submanifold of $M_{i+1}$.  The reason of this assumption is essentially that their direct limit has a natural structure of (n.n.H) convenient manifold. \\
Although each manifold $\mathcal{M}_i$ is open in $\mathcal{E}_i$, since $\mathcal{E}_i$ is a closed subbundle of $\mathcal{E}_j$, it follows that $ \left( \mathcal{M}_i,\chi_i^j \right) _{(i,j)\in\mathbb{N}^2, i \leq j}$ is an ascending sequence of convenient manifolds.
\end{remark}

As in the previous section, for each $i\in \mathbb{N}$, we denote by  $(\mathbf{T}\mathcal{A}_i, \hat{\mathbf{p}}_i, \mathcal{A}_i, \hat{\rho}_i)$ the prolongation  of $\mathcal{A}_i$ over $\mathcal{M}_i=\mathcal{A}_i$ and $[.,.]_{\mathbf{T}\mathcal{A}_i}$ the prolongation of the Lie bracket $[.,.]_{\mathcal{A}_i}$ on projectable section of $\mathbf{T}\mathcal{A}_i$.\\

Adapting the argument used in proof of Proposition  \ref{P_ProjectiveLimitProlongationBracket} to this setting of strong asending sequences and direct limits, we have the result below. Note that, in this context, the prolongation is  not Hausdorff in general. However, all the arguments used in the proofs are local and so they still work in this context.

\begin{proposition}
\label{P_DirectLimitProlongationBracket}
Consider a direct sequence of Banach-Lie algebroid bundles
\begin{center}  
$\left( \left(  \mathcal{A}_{i},\pi_{i},M_{i},\rho_{i},[.,.]_{\mathcal{A}_i} \right) ,
\left( \eta_i^j,\xi_i^j, \varepsilon_i^j \right) \right) _{(i,j)\in\mathbb{N}^2, i \leq j}$ 
\end{center}
(resp. Banach bundles $ \left( \left( \mathcal{E}_{i},\mathbf{\pi}_{i},M_{i},\rho_{i},[.,.]_i \right) , \left( \xi_i^j, \varepsilon_i^j \right) \right) _{(i,j)\in\mathbb{N}^2, i \leq j}$) and a sequence of open fibred manifolds $\mathcal{M}_i $ of $\mathcal{E}_i$  compatible with  algebroid prolongations. Then
\begin{enumerate}
\item 
$\left( \underrightarrow{\lim}\mathbf{T}\mathcal{M}_i,\underrightarrow{\lim}\hat{\mathbf{p}}_i,\underrightarrow{\lim}\mathcal{M}_i,
\underleftarrow{\lim}\hat{\rho}_i \right)  $ is a convenient  anchored bundle which is the prolongation of
$\left( \underrightarrow{\lim}\mathcal{A}_i,\underrightarrow{\lim}\pi_i,M_i \right)$ over $\underrightarrow{\lim}\mathcal{M}_i$.
\item 
Consider any open set $U$ in $M$  and a sequence  of open sets $U_i$ in $M_i$ such that $U=\underrightarrow{\lim}U_i$.  We denote by $\mathfrak{P}^{dl}(\mathbf{T}\mathcal{M}_U)$ the $C^\infty(U)$-module generated by all direct limits $\mathbf{X}=\underrightarrow{\lim}\mathbf{X}_i$  of projectable sections  $\mathbf{X}_i$  of $\mathbf{T}\mathcal{M}_i$ over $\{\mathcal{M}_i\}_{U_i}$.\\
Then there exists a Lie bracket $[.,.]_{\mathbf{T}\mathcal{M_U}} $ defined on  $\mathfrak{P}^{pl}(\mathbf{T}\mathcal{M}_U)$ which  satisfies the assumptions of Definition \ref{D_AlmostLieBracketOnAnAnchoredBundle} characterized by
\[
[\mathbf{X}, \mathbf{X}']_{\mathbf{T}\mathcal{M}_U}=\underrightarrow{\lim}[\mathbf{X}_i,\mathbf{X}'_i]_{\mathbf{T}\mathcal{M}_i}
\]
where  $\mathbf{X}=\underrightarrow{\lim}\mathbf{X}_i$ and  $\mathbf{X}'=\underrightarrow{\lim}\mathbf{X}'_i$.\\
\end{enumerate}
\end{proposition}

As in the context of Projective sequences, $\mathfrak{P}^{dl}(\mathbf{T}\mathcal{M}_U)$ is a submodule of the module $\mathfrak{P}(\mathbf{T}\mathcal{M}_U)$ generated by projectable sections $\mathbf{T}\mathcal{M}_U$. Therefore,  again by analog 
argument as used in the proof of Theorem \ref{T_PartialLieAlgebroid},   the set $\{\mathfrak{P}^{dl}(\mathbf{T}\mathcal{M}_U), U \textrm{ open set  in } M\}$ generates a sheaf of $\mathfrak{P}^{dl}_\mathcal{M}$ over $\mathcal{M}$. Moreover,  for any open set $\mathcal{U}$ in $ \mathcal{M}$, according to Proposition \ref{P_DirectLimitProlongationBracket}, the restriction of $\hat{\rho}$ to each $\mathfrak{P}^{dl}(\mathbf{T}\mathcal{U})$ is a Lie algebra morphism into the Lie algebra morphism into the Lie algebra of vector fields on $\mathcal{U}$. Thus we obtain:

\begin{theorem}
\label{T_DirectLimitOfProlongationOfBanachAnchoredBundles}
Consider a direct sequence of Banach-Lie algebroid  bundles  
\begin{center}
$\left( \left(  \mathcal{A}_{i},\pi_{i},M_{i},\rho_{i},[.,.]_{\mathcal{A}_i} \right),\left( \eta_i^j, \xi_{i}^{j}, \varepsilon_{i}^{j} \right) \right) _{(i,j)\in\mathbb{N}^2, j \geq i}$ 
\end{center}
(resp. Banach bundles $\left( \left(  \mathcal{E}_{i},\mathbf{\pi}_{i},M_{i},\rho_{i},[.,.]_i \right),\left( \xi_i^j, \varepsilon_i^j \right) \right) _{(i,j)\in\mathbb{N}^2, i \leq j}$) and a sequence of open fibred manifolds $\mathcal{M}_i $ of $\mathcal{E}_i$    compatible with  algebroid prolongations.
 Then  $\left( \underrightarrow{\lim}\mathbf{T}\mathcal{M}_i,\underrightarrow{\lim}\hat{\mathbf{p}}_i,\underrightarrow{\lim}\mathcal{M}_i,
\underleftarrow{\lim}\hat{\rho}_i \right)  $ is a convenient  anchored bundle which is the prolongation of
$\left( \underrightarrow{\lim}\mathcal{A}_i,\underrightarrow{\lim}\pi_i,M_i \right)$ over $\underrightarrow{\lim}\mathcal{M}_i$. Moreover $\left(\mathbf{T}\mathcal{M},\hat{\bf p}, \mathcal{M}, \hat{\rho},\mathfrak{P}^{dl}_{\mathcal{M}}\right)$ is a strong  partial convenient Lie algebroid.
\end{theorem}

\appendix

\section{Projective limits}
\label{_ProjectiveLimits}

\subsection{Projective limits of topological spaces}
\label{__ProjectiveLimitsOfTopologicalSpaces}

\begin{definition}
\label{D_ProjectiveSequenceTopologicalSpaces}
A projective sequence of topological spaces\index{projective sequence!of topological spaces} is a sequence\\
 $\left( \left(  X_{i},\delta_{i}^{j}\right) \right)_{(i,j) \in \mathbb{N}^2,\; j \geq i}$ where

\begin{description}
\item[\textbf{(PSTS 1)}]
For all $i\in\mathbb{N},$ $X_{i}$ is a topological space;

\item[\textbf{(PSTS 2)}]
For all $\left( i,j \right)\in\mathbb{N}^2$ such that $j\geq i$,
$\delta_{i}^{j}:X_{j}\to X_{i}$ is a continuous map;

\item[\textbf{(PSTS 3)}]
For all $i\in\mathbb{N}$, $\delta_{i}^{i}={Id}_{X_{i}}$;

\item[\textbf{(PSTS 4)}]
For all $\left( i,j,k \right)\in\mathbb{N}^3$ such that $k \geq j \geq i$, $\delta_{i}^{j}\circ\delta_{j}^{k}=\delta_{i}^{k}$.
\end{description}
\end{definition}

\begin{notation}
\label{N_ProjectiveSequence}
For the sake of simplicity, the projective sequence $\left( \left(  X_{i},\delta_{i}^{j}\right) \right)_{(i,j) \in \mathbb{N}^2,\; j \geq i}$ will be denoted $\left(  X_{i},\delta_{i}^{j} \right) _{j\geq i}$.
\end{notation}

An element $\left(  x_{i}\right)  _{i\in\mathbb{N}}$ of the product
${\displaystyle\prod\limits_{i\in\mathbb{N}}}X_{i}$ is called a \emph{thread}\index{thread} if, for all $j\geq i$, $\delta_{i}^{j}\left(  x_{j}\right)=x_{i}$.

\begin{definition}
\label{D_ProjectiveLimitOfASequence}
The set $X=\underleftarrow{\lim}X_{i}$\index{$X=\underleftarrow{\lim}X_{i}$} of all threads, endowed with the finest topology for which all the projections $\delta_{i}:X\to X_{i} $ are continuous, is called the projective limit of the sequence\index{projective limit!of a sequence} $\left(  X_{i},\delta_{i}^{j} \right) _{j\geq i}$.
\end{definition}

A basis\index{basis!of a topology} of the topology of $X$ is constituted by the subsets $\left( \delta_{i} \right)  ^{-1}\left(  U_{i}\right)  $ where $U_{i}$ is an open subset of $X_{i}$ (and so $\delta_i$ is open whenever $\delta_i$ is surjective).

\begin{definition}
\label{D_ProjectiveSequenceMappings}
Let $\left(  X_{i},\delta_{i}^{j} \right)  _{j\geq i}$ and $\left(  Y_{i},\gamma_{i}^{j} \right)  _{j\geq i}$ be two projective sequences whose respective projective limits are $X$ and $Y$.

A sequence $\left(  f_{i}\right)  _{i\in\mathbb{N}}$ of continuous mappings $f_{i}:X_{i}\to Y_{i}$, satisfying, for all $(i,j) \in \mathbb{N}^2,$ $j \geq i,$ the coherence condition\index{coherence condition}
\[
\gamma_{i}^{j}\circ f_{j}=f_{i}\circ\delta_{i}^{j}%
\]
is called a projective sequence of mappings\index{projective sequence!of mappings}.
\end{definition}

The projective limit of this sequence is the mapping
\[
\begin{array}
[c]{cccc}%
f: & X & \to & Y\\
& \left(  x_{i}\right)  _{i\in\mathbb{N}} & \mapsto & \left(  f_{i}\left(
x_{i}\right)  \right)  _{i\in\mathbb{N}}%
\end{array}
\]

The mapping $f$ is continuous if all the $f_{i}$ are continuous (cf. \cite{AbMa}).

\subsection{Projective limits of Banach spaces}
\label{__ProjectiveLimitsOfBanachSpaces}
Consider a projective sequence $\left(  \mathbb{E}_{i},\delta_{i}^{j} \right)  _{j\geq i}$ of Banach spaces.
\begin{remark}
\label{R_ProjectiveSequenceOfBondingsMapsBetweenBanachSpacesDeterminedByConsecutiveRanks}
Since we have a countable sequence of Banach spaces, according to the properties of bonding maps, the sequence  $\left( \delta_i^j\right)_{(i,j)\in \mathbb{N}^2, \;j\geq i}$ is well defined by the sequence of bonding maps $\left( \delta_i^{i+1}\right) _{i\in \mathbb{N}}$.
\end{remark}

\subsection{Projective limits of differential maps}
\label{__ProjectiveLimitsOfDifferentialMapsBetweenFrechetSpaces}
The following proposition (cf. \cite{Gal}, Lemma 1.2 and \cite{BCP}, Chapter 4) is essential 
\begin{proposition} 
\label{P_ProjectiveLimitsOfDifferentialMaps}
Let $\left( \mathbb{E}_i,\delta_i^j \right) _{j\geq i}$ be a projective sequence of Banach spaces whose projective limit is the Fréchet space $\mathbb{F}=\underleftarrow{lim} \mathbb{E}_i$ and $ \left( f_i : \mathbb{E}_i \to \mathbb{E}_i  \right) _{i \in \mathbb{N}} $ a projective sequence of differential maps whose projective limit is $f=\underleftarrow{\lim} f_i$.
Then the following conditions hold:
\begin{enumerate}
\item
$f$ is smooth in the convenient sense (cf. \cite{KrMi})
\item
For all $x = \left( x_i \right) _{i \in \mathbb{N}}$, $df_x = \underleftarrow{\lim} { \left( df_i \right) }_{x_i} $.
\item
$df = \underleftarrow{\lim}df_i$.
\end{enumerate}
\end{proposition}

\subsection{Projective limits of Banach manifolds}
\label{__ProjectiveLimitsOfBanachManifolds}

\begin{definition}
\label{D_ProjectiveSequenceofBanachManifolds}
The projective sequence $\left( M_{i},\delta_{i}^{j} \right) _{j\geq i}$ is called \textit{projective sequence of Banach manifolds}\index{projective sequence!of Banach manifolds} if
\begin{description}
\item[\textbf{(PSBM 1)}]
$M_{i}$ is a manifold modelled on the Banach space $\mathbb{M}_{i}$;

\item[\textbf{(PSBM 2)}]
$\left(  \mathbb{M}_{i},\overline{\delta_{i}^{j}}\right) _{j\geq i}$ is a projective sequence of Banach spaces;

\item[\textbf{(PSBM 3)}]
For all $x=\left(  x_{i}\right)  \in M=\underleftarrow{\lim}M_{i}$, there exists a projective sequence of local
charts $\left(  U_{i},\xi_{i}\right)  _{i\in\mathbb{N}}$ such that
$x_{i}\in U_{i}$ where one has the relation
\[
\xi_{i}\circ\delta_{i}^{j}=\overline{\delta_{i}^{j}}\circ\varphi_{j};
\]

\item[\textbf{(PSBM 4)}]
 $U=\underleftarrow{\lim}U_{i}$ is a non empty open set in $M$.
\end{description}
\end{definition}

Under the assumptions \textbf{(PSBM 1)} and  \textbf{(PSBM 2)} in Definition \ref{D_ProjectiveSequenceofBanachManifolds}, the assumptions \textbf{(PSBM 3)}] and \textbf{(PSBM 4)}  around $x\in M$ is called \emph{the projective limit chart property} around $x\in M$ and  $(U=\underleftarrow{\lim}U_{i}, \phi=\underleftarrow{\lim}\phi_{i})$ is called a \emph{projective limit chart}.

The projective limit $M=\underleftarrow{\lim}M_{i}$ has a structure of Fr\'{e}chet manifold modelled on the Fr\'{e}chet space $\mathbb{M}
=\underleftarrow{\lim}\mathbb{M}_{i}$ and is called a \emph{$\mathsf{PLB}$-manifold}\index{$\mathsf{PLB}$-manifold}. The differentiable structure is defined \textit{via} the charts $\left(  U,\varphi\right)  $ where $\varphi
=\underleftarrow{\lim}\xi_{i}:U\to\left(  \xi_{i}\left(U_{i}\right)  \right) _{i \in \mathbb{N}}.$\\
$\varphi$ is a homeomorphism (projective limit of homeomorphisms) and the charts changings $\left(  \psi\circ
\varphi^{-1}\right)  _{|\varphi\left(  U\right)  }=\underleftarrow{\lim
}\left(  \left(  \psi_{i}\circ\left(  \xi_{i}\right)  ^{-1}\right)
_{|\xi_{i}\left(  U_{i}\right)  }\right)  $ between open sets of
Fr\'{e}chet spaces are smooth in the sense of convenient spaces.

\subsection{Projective limits of Banach vector bundles }
\label{__ProjectiveLimitsOfBanachVectorBundles}

Let $\left(  M_{i},\delta_{i}^{j}\right)  _{j\geq i}$ be a projective sequence of Banach manifolds where each manifold $M_{i}$ is modelled on the Banach space $\mathbb{M}_{i}$.\\
For any integer $i$, let $ \left( E_{i},\pi_{i},M_{i} \right) $ be the Banach
vector bundle whose type fibre is the Banach vector space $\mathbb{E}_{i}$
where $\left(  \mathbb{E}_{i},\lambda_{i}^{j}\right)  _{j\geq i}$ is a projective sequence of Banach spaces.

\begin{definition}
\label{D_ProjectiveSequenceBanachVectorBundles}
$\left( (E_i,\pi_i,M_i),\left(\xi_i^j,\delta_i^j \right) \right) _{j \geq i}$, where $\xi_i^j:E_j \to E_i$ is a morphism of vector bundles, is called a projective sequence of Banach vector bundles\index{projective sequence!of Banach vector bundles} on the projective sequence of manifolds $\left(  M_{i},\delta_{i}^{j}\right)  _{j\geq i}$ if, for all $ \left( x_{i} \right) $, there exists a projective sequence of
trivializations $\left(  U_{i},\tau_{i}\right)  $ of $\left(  E_{i},\pi
_{i},M_{i}\right) $, where $\tau_{i}:\left(  \pi_{i}\right)  ^{-1}\left(
U_{i}\right)  \to U_{i}\times\mathbb{E}_{i}$ are local
diffeomorphisms, such that $x_{i}\in U_{i}$ (open in $M_{i}$) and where
$U=\underleftarrow{\lim}U_{i}$ is a non empty open set in $M$
 where, for all $(i,j) \in \mathbb{N}^2$ such that $j\geq i,$ we have the compatibility condition
\begin{description}
\item[(\textbf{PLBVB})]
$\left(  \delta_{i}^{j}\times\lambda_{i}^{j}\right)  \circ\tau_{j}=\tau_{i}\circ \xi_i^j$.
\end{description}
\end{definition}

With the previous notations,  $(U=\underleftarrow{\lim}U_{i}, \tau=\underleftarrow{\lim}\tau_i)$   is called a \emph{ projective bundle chart limit}\index{projective bundle chart limit}. The triple of  projective limit
 $(E=\underleftarrow{\lim}E_{i}, \pi=\underleftarrow{\lim}\pi_{i}, M=\underleftarrow{\lim}M_{i}))$ is called a \emph{projective limit of Banach bundles} or $\mathsf{PLB}$-bundle\index{$\mathsf{PLB}$-bundle} for short. \\

The following proposition generalizes the result of \cite{Gal} about the projective limit of tangent bundles to Banach manifolds (cf. \cite{DGV} and \cite{BCP}). 

\begin{proposition}
\label{P_ProjectiveLimitOfBanachVectorBundles}
Let $\left( (E_i,\pi_i,M_i),\left(\xi_i^j,\delta_i^j \right) \right)_{j \geq i}$ be a projective sequence of Banach vector bundles. \\
Then $\left(  \underleftarrow{\lim}E_i,\underleftarrow{\lim}\pi_i,\underleftarrow{\lim}M_i \right)  $ is a Fr\'{e}chet vector bundle.
\end{proposition}

\begin{definition}
\label{D_ProjectiveSequenceofBanachLieAlgebroids}
$\left( \left(  E_{i},\pi_{i},M_{i},\rho_{i}, [.,.]_{i} \right),\left( \xi_i^j, \delta_{i}^{j} \right) \right) _{(i,j)\in\mathbb{N}^2, j \geq i}$ is called a projective sequence of Lie algebroids \index{projective sequence!of Lie algebroids} if
\begin{description}
\item[\textbf{(PSBLA 1)}]
$\left(  E_{i},\xi_i^j\right)  _{j\geq i} $ is a 
projective sequence of Banach vector bundles ($\pi_{i}:E_{i}\to M_{i})_{i\in\mathbb{N}}$ over the 
projective sequence of manifolds $ \left(  M_{i},\delta_{i}^{j}\right) _{j\geq i}$;

\item[\textbf{(PSBLA 2)}]
For all $\left( i,j \right) \in\mathbb{N}^2$ such that $j\geq i$, one has
\[
\rho_{i}\circ \xi_i^j=T\delta_{i}^{j}\circ\rho_{j}
\]
\item[\textbf{(PSBLA 3)}]
$\xi_i^j:E_{j}\to E_{i}$ is a Lie morphism  from $\left(  E_{j},\pi_{j},M_{j},\rho_{j} \right)  $ to $\left(  E_{i},\pi_{i},M_{i},\rho_{i}\right)  $
\end{description}
\end{definition}

We then have the following result (cf. \cite{BCP}):
\begin{theorem}
\label{T_ProjectiveLimitOfBanachLieAlgebroids} 
Let $\left( \left(  E_{i},\pi_{i},M_{i},\rho_{i}, [.,.]_{i} \right),\left( \xi_i^j, \delta_{i}^{j} \right) \right) _{(i,j)\in\mathbb{N}^2, j \geq i}$ be a projective sequence of Banach-Lie algebroids.  If $(M_i,\delta_i^j)_{(i,j)\in\mathbb{N}^2, j \geq i}$ is a submersive projective sequence,
then $\left(\underleftarrow{\lim}E_{i},\underleftarrow{\lim}\pi_{i},\underleftarrow{\lim}M_{i},\underleftarrow{\lim}\rho_{i}\right)  $ is a strong partial Fr\'{e}chet Lie algebroid.
\end{theorem}

\section{Direct limits}
\label{_Directlimits}

\subsection{Direct limits of topological spaces}
\label{__DirectLimitsOfTopologicalSpaces}

Let $\left\{  \left(  Y_{i},\varepsilon_{i}^{j}\right)  \right\}  _{(i,j)\in I^2,\ i \preccurlyeq j}$ be a direct system of topological spaces and continuous maps. The direct limit $\left\{  \left(  X,\varepsilon_{i}\right)
\right\}  _{i\in I}$ of the sets becomes the direct limit in the category ${\bf TOP}$ of topological spaces if $X$ is endowed with the direct limit topology (DL-topology for short)\index{DL-topology}\index{topology!DL-}, i.e. the final topology with respect to the inclusion maps $\varepsilon_{i}:X_{i}\to X$ which corresponds to the finest topology which makes the maps $\varepsilon_{i}$ continuous. So $O\subset X$ is open if and only if $\varepsilon_{i}^{-1}\left(  O\right)  $ is open in $X_{i}$ for each
$i\in I$.
\begin{definition}
\label{D_AscendingSequenceOfTopologicalSpaces}
\index{ascending sequence!of topological spaces}
Let $\mathcal{S}=\left(   \left(  X_{n},\varepsilon_{n}^{m}\right)  \right) _{(m,n) \in\mathbb{N}^2,\ n\leq m}$ be a direct sequence of topological spaces such that each $\varepsilon_{n}^{m}$ is injective. Without loss of generality, we may assume that we have
\[
X_{1}\subset X_{2}\subset\cdots\subset X_{n}\subset X_{n+1}\subset\cdots
\]
and $\varepsilon_{n}^{n+1}$ becomes the natural inclusion.

\begin{description}
\item[\textbf{(ASTS)}]
$\mathcal{S}$ will be called an ascending sequence of topological spaces.
\item[\textbf{(SASTS)}]
Moreover, if each $\varepsilon_{n}^{m}$ is a topological embedding, then we will say that $\mathcal{S}$ is a \emph{strict ascending sequence of topological spaces} (\textit{expanding sequence}
\end{description}
\end{definition}

\begin{notation}
\label{N_DirectSequence}
The direct sequence $\left( \left(  X_{n},\varepsilon_{n}^{m}\right) \right)_{(n,m) \in \mathbb{N}^2,\; n \leq m}$ will be denoted $\left(  X_{n},\varepsilon_{n}^{m} \right) _{n \leq m}$.
\end{notation}

If $\left(  X_{n},\varepsilon_{n}^{m}\right) _{n\leq m}$ is a strict ascending sequence of topological spaces\index{ascending sequence!of topological spaces}, each $\varepsilon_{n}$ is a topological embedding from $X_{n}$ into $X=\underrightarrow{\lim}X_{n}$.

\subsection{Direct limits of ascending sequences of Banach manifolds}
\label{__DirectLimitsOfAscendingSequencesOfBanachManifolds}
\index{direct limit!of Banach manifolds}

\begin{definition}
\label{D_AscendingSequenceOfBanachManifolds}
$\mathcal{M}=(M_{n},\varepsilon_{n}^{n+1})  _{n\in\mathbb{N}}$ is called an ascending sequence of Banach manifolds if, for any $n\in\mathbb{N}$, $\left(  M_{n},\varepsilon_{n}^{n+1}\right)  $ is a closed  submanifold of
$M_{n+1}$.
\end{definition}

\begin{proposition}
\label{P_ConditionsDefiningDLCP}(cf. \cite{CaPe} \cite{BCP})
Let $\mathcal{M}=\left( M_{n},\varepsilon_{n}^{n+1} \right) _{n\in\mathbb{N}}$ be an ascending sequence of Banach manifolds.\\
Assume that for $x\in M=\underrightarrow {\lim}M_{n}$, there exists a sequence of charts $\left( (U_{n},\phi_{n})\right) _{n \in \mathbb{N}}$ of $\left( M_{n}\right) _{n \in \mathbb{N}}$, such that:
\begin{description}
\item[\textbf{(ASC 1)}]
$(U_{n})_{n\in\mathbb{N}}$ is an ascending sequence of chart domains;
\item[\textbf{(ASC 2)}]
{\hfil $\forall n\in\mathbb{N},\ \phi_{n+1}\circ\varepsilon_{n}^{n+1}=\iota_{n}^{n+1}\circ\phi_{n}$.}
\end{description}
Then $U=\underrightarrow{\lim}U_{n}$ is an open set of $M$ endowed with the
$\mathrm{DL}$-topology and $\phi=\underrightarrow{\lim}\phi_{n}$ is a well
defined map from $U$ to $\mathbb{M}=\underrightarrow{\lim}\mathbb{M}_{n}$.\\
Moreover, $\phi$ is a continuous homeomorphism from $U$ onto the open set $\phi(U)$ of $\mathbb{M}$.
\end{proposition}

\begin{definition}
\label{D_DLChartProperty}
We say that an ascending sequence $\mathcal{M}= (M_{n},\varepsilon_{n}^{n+1}) _{n\in\mathbb{N}}$ of Banach manifolds has the direct limit chart property \emph{\textbf{(DLCP)}}\index{direct!limit chart} at $x$ if it satisfies both \emph{\textbf{(ASC 1)}} and \emph{\textbf{(ASC 2)}}.
\end{definition}

We then have the fundamental result (cf. \cite{CaPe}).
\begin{theorem}
\label{T_LBCmanifold}
Let $ \left( M_{n} \right) _{n\in\mathbb{N}}$ be an ascending sequence of Banach $C^\infty$-manifolds, modelled on the Banach spaces $\mathbb{M}_{n}$. Assume that
\begin{description}
\item[\textbf{(ASBM 1)}]
$ \left( M_{n} \right) _{n\in\mathbb{N}}$ has the direct limit chart property \emph{\textbf{(DLCP)}} at each point $x\in M=\underrightarrow{\lim}M_{n}$
\item[\textbf{(ASBM 2)}]
$\mathbb{M}=\underrightarrow{\lim}\mathbb{M}_{n}$ is a convenient space.
\end{description}
Then there is a unique  n.n.H. convenient manifold structure on
$M=\underrightarrow{\lim}M_{n}$ modelled on the convenient space $\mathbb{M}$ such that the topology associated to this structure is the  $DL$-topology on $M$. \\
In particular, for each $n\in\mathbb{N}$, the canonical injection $\varepsilon_{n}:M_{n}\longrightarrow M$ is an
injective conveniently smooth map and $(M_{n},\varepsilon_{n})$ is a closed submanifold of $M$.\\
Moreover, if each $M_n$ is locally compact or is open in $M_{n+1}$ or is a paracompact Banach manifold closed  in $M_{n+1}$, then $M=\underrightarrow{\lim}M_{n}$ is provided with a Hausdorff  convenient manifold structure.\\
\end{theorem}

\subsection{Direct limits of Banach vector bundles}
\label{__DirectLimitsOfBanachVectorBundles}
\index{direct limit!of Banach vector bundles}%

\begin{definition}
\label{D_AscendingSequenceBanachVectorBundles}$\left( (E_n,\pi_n,M_n),\left(\lambda_n^{n+1},\varepsilon_n^{n+1} \right) \right)_{n \in \mathbb{N}}$ is called a strong ascending sequence of Banach vector bundles if the following assumptions are satisfied:

\begin{description}
\item[\textbf{(ASBVB 1)}]
$\mathcal{M}=(M_{n})_{n\in\mathbb{N}}$ is an
ascending sequence of Banach $C^\infty$-manifolds, where $M_{n}$ is modelled
on the Banach space $\mathbb{M}_{n}$ such that $\mathbb{M}_{n}$ is a
supplemented Banach subspace of $\mathbb{M}_{n+1}$ and the inclusion
$\varepsilon_{n}^{n+1}:M_{n}\to M_{n+1}$ is a $C^\infty$
injective map such that $(M_{n},\varepsilon_{n}^{n+1})$ is a closed submanifold of $M_{n+1}$;
\item[\textbf{(ASBVB 2)}]
The sequence $(E_{n})_{n\in\mathbb{N}}$ is an ascending sequence such that the sequence of typical fibres $\left(\mathbb{E}_{n}\right)  _{n\in\mathbb{N}}$ of $(E_{n})_{n\in\mathbb{N}}$ is an ascending sequence of Banach spaces and $\mathbb{E}_{n}$ is a supplemented Banach subspace of $\mathbb{E}_{n+1}$;
\item[\textbf{(ASBVB 3)}]
For each $n\in\mathbb{N}$, $\pi_{n+1}\circ\lambda_{n}^{n+1}=\varepsilon_{n}^{n+1}\circ\pi_{n}$ where $\lambda
_{n}^{n+1}:E_{n}\to E_{n+1}$ is the natural inclusion;
\item[\textbf{(ASBVB 4)}]
Any $x\in M=\underrightarrow{\lim}M_{n}$ has the
direct limit chart property \emph{\textbf{(DLCP) }}for $(U=\underrightarrow
{\lim}U_{n},\phi=\underrightarrow{\lim}\phi_{n})$;
\item[\textbf{(ASBVB 5)}]
For each $n\in\mathbb{N}$, there exists a
trivialization $\Psi_{n}:\left(  \pi_{n}\right)  ^{-1}\left(  U_{n}\right)
\to U_{n}\times\mathbb{E}_{n}$ such that, for any $i\leq j$, the
following diagram is commutative:
\end{description}
\[
\xymatrix{
\left( \pi_{i} \right) ^{-1} \left( U_{i} \right) \ar[r]^{\lambda_{i}^{j}}\ar[d]_{\Psi_{i}}
                                        & \left( \pi_{j} \right) ^{-1} \left( U_{j} \right) \ar[d]^{\Psi_j}\\
U_{i} \times \mathbb{E}_{i} \ar[r]^{\varepsilon_{i}^{j} \times \iota_{i}^{j}}
                                        & U_{j} \times \mathbb{E}_{j}
}
\]
\end{definition}

For example, the sequence $\left( \left( TM_n,\pi_n,M_n \right) , \left( T\varepsilon_n^{n+1},\varepsilon_n^{n+1} \right) \right) _{n \in \mathbb{N}}$ is a strong ascending sequence of Banach vector bundles whenever $ \left( M_{n} \right)_{n\in\mathbb{N}}$ is an ascending sequence which has the direct limit chart property at each point of $x \in M=\underrightarrow{\lim}M_{n}$ whose model $\mathbb{M}_{n}$ is supplemented in $\mathbb{M}_{n+1}$.

\begin{notation}
From now on and for the sake of simplicity, the strong ascending sequence of vector bundles $\left( (E_n,\pi_n,M_n),\left(\lambda_n^{n+1},\varepsilon_n^{n+1} \right) \right)_{n \in \mathbb{N}}$ will be denoted $\left(  E_{n},\pi_{n},M_{n}\right)_{\underrightarrow{n}}$.
\end{notation}

We then have the following result given in \cite{CaPe}.
\begin{proposition}
\label{P_StructureOnDirectLimitOfLinearBundles}
Let $\left( E_{n},\pi_{n},M_{n}\right) _{\underrightarrow{n}}$ be a strong
ascending sequence of Banach vector bundles. We have:
\begin{enumerate}
\item
$\underrightarrow{\lim}E_{n}$ has a structure of not necessarily Hausdorff convenient manifold modelled on the LB-space $\underrightarrow{\lim}\mathbb{M}_{n}\times\underrightarrow{\lim}\mathbb{E}_{n}$ which has a
Hausdorff convenient structure if and only if $M$ is Hausdorff.
\item
$\left(  \underrightarrow{\lim}E_{n},\underrightarrow{\lim}\pi
_{n},\underrightarrow{\lim}M_{n}\right)  $ can be endowed with a structure of
convenient vector bundle whose typical fibre is $\underrightarrow{\lim
}\mathbb{\mathbb{E}}_{n}$.
\end{enumerate}
\end{proposition}

\end{document}